\documentclass[11pt,a4paper]{article}

\usepackage[utf8]{inputenc}
\usepackage[T1]{fontenc}
\usepackage{lmodern}

\usepackage{amsmath,amsfonts,amssymb,amsthm}
\usepackage{mathtools}
\usepackage{accents}

\usepackage{hyperref}
\hypersetup{
    colorlinks=true,
    linkcolor=blue,
    citecolor=red,
    urlcolor=blue,
    pdfborder={0 0 0}
}

\usepackage[a4paper]{geometry}

\usepackage{cleveref}
  \crefname{section}{Section}{Sections}
  \crefname{figure}{Figure}{Figures}
  \crefname{theorem}{Theorem}{Theorems}
  \crefname{lemma}{Lemma}{Lemmas}
  \crefname{proposition}{Proposition}{Propositions}  
  \crefname{corollary}{Corollary}{Corollaries}
  \crefname{definition}{Definition}{Definitions}
  \crefname{example}{Example}{Examples}
  \crefname{remark}{Remark}{Remarks}
  \crefname{equation}{}{}

\newtheorem{theorem}{Theorem}[section]
\newtheorem{lemma}[theorem]{Lemma}
\newtheorem{proposition}[theorem]{Proposition}
\newtheorem{corollary}[theorem]{Corollary}

\theoremstyle{definition}

\newtheorem{remark}[theorem]{Remark}

\DeclarePairedDelimiter\abs{\lvert}{\rvert}
\DeclarePairedDelimiter\norm{\lVert}{\rVert}

\newcommand*{\defeq}{\mathrel{\mathop:}=}

\newcommand*{\mmiddle}[1]{\mathrel{}\middle#1\mathrel{}}

\renewcommand*{\Re}{\operatorname{Re}}
\renewcommand*{\Im}{\operatorname{Im}}

\newcommand*{\diam}{\operatorname{diam}}
\newcommand*{\dist}{\operatorname{dist}}
\newcommand*{\crad}{\operatorname{crad}}
\newcommand*{\hcap}{\operatorname{hcap}}

\newcommand*{\cF}{\mathcal{F}}

\newcommand*{\bbC}{\mathbf{C}}

\newcommand*{\bbE}{\mathbf{E}}
\newcommand*{\bbH}{\mathbf{H}}
\newcommand*{\bbN}{\mathbf{N}}
\newcommand*{\bbP}{\mathbf{P}}
\newcommand*{\bbR}{\mathbf{R}}

\newcommand*{\barH}{\overline{\bbH}}

\newcommand*{\sle}[1]{SLE$_{#1}$}
\newcommand*{\slek}{\sle{\kappa}}

\newcommand*{\logp}{\log^*}

\title{Refined regularity of SLE}

\author{
Yizheng Yuan\thanks{TU Berlin, Germany \& U Cambridge, United Kingdom.
Email: \texttt{yy547@cam.ac.uk}}
}

\begin{document}

\maketitle

\begin{abstract}
We prove refined (variation and H\"older-type) regularity statements for the SLE trace (under capacity parametrisation). More precisely, we show that the trace has finite $\psi$-variation for $\psi(x) = x^d(\log 1/x)^{-d-\varepsilon}$ and H\"older-type modulus $\varphi(t) = t^\alpha(\log 1/t)^{\beta}$ where $d$ and $\alpha$ are the optimal $p$-variation and H\"older exponents of SLE$_\kappa$ which have been previously identified by Viklund, Lawler (2011) and Friz, Tran (2017). For SLE$_8$, we simplify a step in the proof by Kavvadias, Miller, and Schoug (2021), and get the modulus $\varphi(t) = (\log 1/t)^{-1/4}(\log\log 1/t)^{2+\varepsilon}$.

Finally, for $\kappa \ge 8$, we prove regularity estimates for the uniformising maps that hold uniformly in time, namely $\sup_t |\hat f_t'(u+iv)| \lesssim v^{2\alpha-1}(\log 1/v)^\beta$ in case $\kappa>8$ and $v^{-1}(\log 1/v)^{-1/4}(\log\log 1/v)^{1+\varepsilon}$ in case $\kappa=8$.

Our results are obtained from analysing the forward Loewner differential equation (in contrast to the other mentioned works which analyse the backward equation).
\end{abstract}

\section{Introduction}

Schramm-Loewner evolution (SLE) is a family of random curves that appear naturally in conformally invariant models in the plane. First introduced by O.\@ Schramm to describe the scaling limits of the loop-erased random walk and the uniform spanning tree, they have been shown to appear in the scaling limits of many more models such as the Ising model and Bernoulli percolation. Moreover, they are deeply connected to other conformally invariant objects such as Brownian motion, Gaussian free field, and Liouville quantum gravity.

Regularity of the SLE trace has been studied by many authors, starting from \cite{rs-sle}. From the works \cite{vl-sle-hoelder, bef-sle-dimension, ft-sle-regularity} (and earlier \cite{wer-sle-pvar}) we know the optimal $p$-variation and Hölder exponents which are $p_* = d = (1+\kappa/8) \wedge 2$ and $\alpha_* = 1-\frac{\kappa}{24+2\kappa-8\sqrt{8+\kappa}}$. More precisely, the trace has finite $p$-variation for $p > d$ and infinite $p$-variation for $p < d$. And (under capacity parametrisation) it is Hölder continuous with exponent $\alpha < \alpha_*$ and not Hölder continuous of exponent $\alpha > \alpha_*$. As for the critical exponents, we do not know but expect that the traces do not have these regularities. This leads to the question of finding the correct modulus. As a comparison, the optimal variation regularity of Brownian motion is shown in \cite{tay-bm-variation} to be $\psi(x) = x^2(\logp\logp(1/x))^{-1}$, and the optimal modulus of continuity is $\varphi(t) = t^{1/2}(\logp(1/t))^{1/2}$ (cf.\@ \cite{Lev37}).

Variation regularity seems more natural in the context of SLE since it is parametrisa\-tion-independent, and for many applications authors only care about SLE as a curve and not about its parametrisation. It corresponds to the best modulus that one can get from any parametrisation of the curve. It naturally gives an upper bound on its Hausdorff dimension which turns out to be its true Hausdorff dimension (cf.\@ \cite{bef-sle-dimension}). The capacity parametrisation arises from analysis via the Loewner equation, but does not enjoy very good regularity. In fact, the natural parametrisation introduced in \cite{ls-natural-parametrisation, lz-natural-parametrisation} gives much better regularity. At least on the Hölder scale, it has the optimal regularity that matches the variation exponent. Namely, it is shown in \cite{zhan-hoelder, ghm-kpz-sle} that (some variant of) SLE in natural parametrisation is Hölder continuous for any exponent $\alpha < 1/d$. The arguments in our paper also suggest a close link between $p$-variation and Minkowski content (which is the natural parametrisation, cf.\@ \cite{lr-minkowski-content}). We leave it to future work to explore further this connection.

\emph{Update:} It was later proved in \cite{hy-sle-regularityx} that the natural parametrisation does not give the best possible modulus of continuity.

The regularity of the domain given by the complement of the SLE trace has also been studied by many authors including \cite{rs-sle, kang-sle-boundary, bs-aims-sle, gms-sle-multifractal, kms-sle48x}. This can be expressed by the regularity of the conformal map that maps the upper half-plane to the domain. It has been shown that at any fixed time, this map is Hölder continuous when $\kappa\neq 4$, and has logarithmic modulus of continuity when $\kappa=4$. In this paper, we prove a similar result when the time is not fixed but the supremum over all times is considered. (As we will explain, this is only relevant in the space-filling regime $\kappa \ge 8$.)

The first results of this paper are about refined variation and Hölder-type regularities of the \slek{} trace.

\begin{theorem}\label{thm:main_psivar}
Let $\kappa \neq 8$. The \slek{} trace $\gamma$ on $[0,T]$ almost surely has finite $\psi$-variation where $\psi(x) = x^{d} (\logp 1/x)^{-d-\varepsilon}$, $d = (1+\kappa/8) \wedge 2$.
\end{theorem}

We also show that the $\psi$-variation constant has finite moments, see \cref{thm:gvar}.

\begin{corollary}
Let $\kappa \neq 8$. The \slek{} trace can be parametrised such that
\[ \abs{\tilde\gamma(t)-\tilde\gamma(s)} \le C\abs{t-s}^{1/d}\left(\logp\frac{1}{\abs{t-s}}\right)^{1+\varepsilon} \]
where $d = (1+\kappa/8) \wedge 2$.
\end{corollary}

Throughout this paper, we let $\ell\colon {[1,\infty[} \to \bbR^+$ denote a non-decreasing function such that $\lim_{s \to \infty} \frac{\ell(2s)}{\ell(s)} = 1$ and
\begin{equation}\label{eq:ell_def}
\int_1^\infty \frac{ds}{s\ell(s)} < \infty .
\end{equation}
Typical examples are $\ell(s) = (\logp s)^{1+\varepsilon}$ or $\ell(s) = (\logp s)(\logp\logp s)\cdots(\logp\cdots\logp s)^{1+\varepsilon}$.

\begin{theorem}\label{thm:main_hoelder}
Let $\kappa \neq 8$. For any fixed $t_0 > 0$, the \slek{} trace $\gamma$ (parametrised by half-plane capacity) restricted to $[t_0,T]$ almost surely satisfies
\[
\abs{\gamma(t)-\gamma(s)} \le C\abs{t-s}^{\alpha} \left(\logp\frac{1}{\abs{t-s}}\right)^{\beta} \ell\left(\logp\frac{1}{\abs{t-s}}\right)^{\beta}
\]
where $\alpha = 1-\frac{\kappa}{24+2\kappa-8\sqrt{8+\kappa}}$, $\beta = \frac{\kappa}{(12+\kappa)\sqrt{8+\kappa}-4(8+\kappa)}$, and $\ell$ as in \eqref{eq:ell_def}.

In case $\kappa > 1$, the same is true for the \slek{} trace on $[0,T]$.

In case $\kappa \le 1$, there exists $\beta < 1$ such that the \slek{} trace on $[0,T]$ almost surely satisfies
\[
\abs{\gamma(t)-\gamma(s)} \le C\abs{t-s}^{1/2} \left(\logp\frac{1}{\abs{t-s}}\right)^{\beta} .
\]
\end{theorem}

We also show that the random factor $C$ has finite moments, see \cref{thm:hoelder}.

\begin{remark}
We actually show that the supremum over partitions $t_0<t_1<...<t_r=T$ of
\[ \sup \sum_j \left( \frac{\abs{\gamma(t_j)-\gamma(t_{j-1})}}{\abs{t_j-t_{j-1}}^{\alpha}(\logp\frac{1}{\abs{t_j-t_{j-1}}})^{\beta} \ell(\logp\frac{1}{\abs{t_j-t_{j-1}}})^{\beta}
} \right)^{p} \]
is almost surely finite for suitable $p > 1$.

Moreover, we can interpolate between this statement and \cref{thm:main_psivar} and have a much wider range of regularity statements.
\end{remark}

\begin{remark}
The boundary effect at $t=0$ is a feature of the half-plane capacity parametrisation, and was already present in the results of \cite{vl-sle-hoelder}. Indeed, \cite[Lemma 3.4]{vl-sle-hoelder} implies that at $t=0$ we cannot have better than $1/2$-Hölder regularity. A way of describing the regularity of the full path with singularity at $0$ was introduced in \cite{bfg-singular-path}.

Our proof gives also an explicit exponent $\beta$ in the case with boundary effect, but we have not attempted to optimise it.
\end{remark}

\noindent\emph{Update:} It was later shown in \cite{hy-sle-regularityx} that the correct $\psi$-variation (at least on segments away from the boundary) is $\psi(x) = x^d(\log\log 1/x)^{-(d-1)}$. The proof in \cite{hy-sle-regularityx} uses some deeper results specific to SLE (such as natural parametrisation, whole-plane reversibility, duality, etc.), and does not immediately transfer to points close to the boundary. The proofs in this paper are obtained directly from the Loewner equation, and do not put limitations on points near the boundary. Also, they could in principle be adapted to other driving functions such as studied in \cite{gw-alpha-sle, cr-le-symmetric-stable, msy-le-semimartingalex, ps-le-levyx}.

\medskip

We do not know whether our exponent $\beta$ in \cref{thm:main_hoelder} is optimal. To prove similar lower bounds, one would need sharp decorrelation results for the SLE flows started from different points. From the arguments in our paper (with some additional refinements) one can recover that the exponents $d$ and $\alpha$ are optimal, but it seems technically more difficult to push the lower bounds to the more refined scale.

Our results so far concern \slek{} with $\kappa \neq 8$. The case $\kappa = 8$ is much more challenging because the trace (under capacity parametrisation) has very bad regularity. Recently, it was shown by \cite{kms-sle48x} that the regularity of SLE$_8$ is given by the modulus $\varphi(t) = (\logp(\frac{1}{t}))^{-1/4+\varepsilon}$. Although we do not know how to give a completely elementary proof of this fact, we are able to simplify a step in their proof. We show that \cite[Lemma 7.1]{kms-sle48x} contains all the non-elementary information on the regularity of SLE$_8$. More precisely, taking the lemma for granted, we can deduce the result by elementary arguments of our paper.

\begin{theorem}[{\cite[Theorem 1.3]{kms-sle48x}}]\label{thm:main_sle8}
The SLE$_8$ trace $\gamma$ on $[0,T]$ (parametrised by half-plane capacity) almost surely satisfies
\[ \abs{\gamma(t)-\gamma(s)} \le C \left(\logp\frac{1}{\abs{t-s}}\right)^{-1/4+\varepsilon} . \]
\end{theorem}

Moreover, we show in \cref{se:sle8} that the random factor $C$ has finite $(2+\delta)$th moment.

\begin{remark}
In fact, we have the slightly stronger result that (at least on segments away from the boundary)
\[
\abs{\gamma(t)-\gamma(s)} \le 
C \left(\logp\frac{1}{\abs{t-s}}\right)^{-1/4}\left(\logp\logp\frac{1}{\abs{t-s}}\right)^{2}\ell\left(\logp\logp\frac{1}{\abs{t-s}}\right)^{1/2} .
\]
This is obtained by using a stronger version of \cite[Lemma 7.1]{kms-sle48x}. Indeed, following the proof of \cite[Proposition 3.4]{ghm-kpz-sle} (see also \cite[Theorem 1.7]{hy-sle-regularityx}), one sees that with probability tending to $1$ as $r \searrow 0$, every segment $\gamma[s,t]$ (at least those away from the boundary) fills a ball of radius comparable to $\abs{\gamma(t)-\gamma(s)}/(\log\frac{1}{\abs{\gamma(t)-\gamma(s)}})$ whenever $\abs{\gamma(t)-\gamma(s)} \le r$. Using this, the argument in \cref{se:sle8} show the refined regularity statement.

We have not made an effort to extend the result to points near the boundary, or to bound the moments of the random constant $C$ in the refined statement.
\end{remark}

Our last main result concerns the uniform-in-time regularity of the uniformising map for \slek{}, $\kappa \ge 8$.

\begin{theorem}\label{th:main_unif_map}
For $\kappa \ge 8$, the \slek{} uniformising maps $\hat f_t\colon \bbH \to \bbH \setminus \gamma[0,t]$ almost surely satisfy
\[ \sup_{t \in [0,T]} \abs{\hat f_t'(w)} \le Ch(\Im w) \]
for $w \in \bbH$ with $\Im w \le 1$
where
\[
h(v) =
\begin{cases}
v^{-1}(\logp(\frac{1}{v}))^{-1/4}(\logp\logp(\frac{1}{v}))(\ell(\logp\logp(\frac{1}{v})))^{1/2} & \text{for }\kappa=8 ,\\
v^{2\alpha-1}(\logp(\frac{1}{v}))^\beta \ell(\logp(\frac{1}{v}))^{\beta} & \text{for }\kappa>8 ,
\end{cases}
\]
and $\alpha = 1-\frac{\kappa}{24+2\kappa-8\sqrt{8+\kappa}}$, $\beta = \frac{\kappa}{(12+\kappa)\sqrt{8+\kappa}-4(8+\kappa)}$, and $\ell$ as in \eqref{eq:ell_def}.
\end{theorem}

We also show that the random factors $C$ have finite moments, see \cref{se:unif_map} for precise statements.

\begin{remark}
In our result, only the case $\kappa \ge 8$ is relevant. In the case $\kappa \le 4$, the \slek{} trace is simple, so the uniform-in-time regularity is the same as the regularity at a given time. In the case $4 < \kappa < 8$, the \slek{} trace swallows points, hence the regularity of the uniformising map is not uniformly bounded in time.
\end{remark}

\begin{remark}
That the regularity of the trace in half-plane capacity parametrisation is linked to the regularity of the uniformising maps is not a coincidence. Typically, the increase of half-plane capacity $\hcap(\gamma[0,t])-\hcap(\gamma[0,s])$ is approximately $v^2$ where $v=\operatorname{height}(\hat f_s^{-1}(\gamma[s,t]))$. On the other hand, picking a point $u+iv \in \hat f_s^{-1}(\gamma[s,t])$, the diameter of $\gamma[s,t]$ is typically of order $v\abs{\hat f_s'(iv)}$. This means that (ignoring large distortions near the boundary) typically $\diam(\gamma[s,t]) \approx \abs{t-s}^{1/2}\abs{\hat f_s'(i\abs{t-s}^{1/2})}$.
\end{remark}

Existing proofs of SLE regularity in the literature commonly analyse the backward SLE flow. Our argument, in contrast, analyses the forward SLE flow (it is also used in \cite{msy-le-semimartingalex}). The key observation is \cref{thm:fw_grid} which allows us to estimate the derivative of the uniformising map by considering the forward flows started from a well chosen grid of points. This allows us to pick suitable stopping times when information on our processes is needed. As a consequence, all our estimates are uniform in time. This turns out also advantageous in the setup of \cite{ps-le-levyx}.

Another ingredient of our proofs is that we parametrise the flow by conformal radius. This parametrisation has been used e.g.\@ in \cite{lw-sle-greens} to study the distance of points to the trace. As we will recall in \cref{se:prelim_general}, increments are naturally bounded by certain conformal radii. Arranging the points by conformal radius reduces redundancies in our bounds. Another feature of this parametrisation is that an important quantity, the angle process of the flow, becomes a radial Bessel process which is easy to analyse.

The idea goes roughly as follows: As we will recall in \cref{rm:crad_vs_path}, estimating increments $\abs{\gamma(t)-\gamma(s)}$ essentially boils down to estimating $\crad(z,H_s) \asymp \dist(z,\gamma[0,s])$ where $z = \hat f_s(iv)$ and $v = \abs{t-s}^{1/2}$. We reverse this procedure by fixing $z$ and running time until $\crad(z,H_s) = \delta$. This tells us that for some $s<t$ we have $\abs{\gamma(t)-\gamma(s)} \approx \delta$. Then we estimate $v = \Im g_s(z)$ to find out $\abs{t-s} \approx v^2$.

It may look at first sight that considering a grid of starting points seems to take too much unnecessary information into account, but in fact all information is being used. Indeed, a positive portion of points in a bounded set do get close to the curve. Each of those points are either ``missed'' by the curve or come close to the tip at some point. The probability of ``missing'' is being considered as the probability of a radial Bessel process hitting the boundary before a given time. For the points that come close to the tip, the information on their one-point martingales are relevant for the regularity of the trace.

It is reasonable that the approach in this paper can also be used to prove (and maybe improve) results in the case of (more regular) deterministic driving functions, analysed e.g.\@ in \cite{rtz-le-slit, fs-le-rp, stw-finite-var}.

In \cref{se:preliminaries}, we summarise a few preliminaries and basic results. We introduce generalised variation, (chordal) Loewner chains and SLE, and collect a few results on radial Bessel processes that we use later. In \cref{se:pf_warmup}, we explain the main idea behind our proofs by showing \slek{}, $\kappa\neq 8$, has a continuous trace. In \cref{se:pf_setup}, we discuss the setup for our later proofs. In the \cref{se:gvar_pf,se:hoelder_pf}, we prove our main \cref{thm:main_psivar,thm:main_hoelder}. In \cref{se:sle8} we explain how apply our arguments to SLE$_8$ and reprove \cref{thm:main_sle8}. Finally, in \cref{se:unif_map}, we prove \cref{th:main_unif_map}.

\subsection*{Notation}

Throughout the paper, we denote by $\bbH$ the upper half-plane $\{z \in \bbC\mid \Im z > 0\}$, and by $\barH$ the closed upper half-plane $\{z \in \bbC\mid \Im z \ge 0\}$. We write $B(z,r)$ for the open ball with radius $r>0$ around $z$. We denote conformal radius by $\crad$ and half-plane capacity by $\hcap$.

We write $x_+ = x \vee 0$ and $\logp(x) = (\log x) \vee 1$. We write $a \lesssim b$ meaning $a \le cb$ for some constant $c < \infty$ that may depend on the context, and $a \asymp b$ meaning $a \lesssim b$ and $b \lesssim a$. Moreover, we write $a \asymp_c b$ to state explicitly the constant $c$.

\medskip
\textbf{Acknowledgements:}
I acknowledge partial support from European Research Council through Consolidator Grant 683164 (PI: Peter Friz) during the stage at TU Berlin where the first preprint was completed, and Starting Grant 804166 (SPRS; PI: Jason Miller) during the stage at U Cambridge where the paper was revised and extended. I thank Peter Friz, Nina Holden, and Konstantinos Kavvadias for helpful discussions and comments. I thank the referees for their careful reading and their comments on an earlier version of the paper.

\section{Preliminaries}
\label{se:preliminaries}

\subsection{Generalised variation}
\label{se:prelim_psi_var}

We summarise the most important facts about $\psi$-variation. See e.g.\@ \cite[Section 5.4]{fv-rp-book} for more details.

Let $\psi\colon {[0,\infty[} \to {[0,\infty[}$ be a homeomorphism. Let $I \subseteq \bbR$ be an interval and $x\colon I \to E$ a function with values in a metric space $(E,d)$. Define
\[ V^M_{\psi;I}(x) = \sup \sum_j \psi\left(\frac{d(x(t_{j+1}),x(t_j))}{M}\right) \]
where the supremum is taken with respect to all finite subsets $\{t_0 < t_1 < ... < t_r\}$ of $I$. Usually $V^1_{\psi;I}(x)$ is called the total $\psi$-variation of $x$ on $I$.

The $\psi$-variation constant of $x$ is defined as
\[ [x]_{\psi\text{-var};I} = \inf\{ M > 0 \mid V^M_{\psi;I}(x) \le 1 \} . \]
In case $E$ is a normed space and $\psi$ is convex, this defines a semi-norm.
Note that in case $\psi(x) = x^p$, this agrees with the notion of $p$-variation.

We say that $\psi$ satisfies the condition ($\Delta_c$) if\footnote{This condition is a bit stronger than what is necessary, but will suffice for our purposes.}
\begin{itemize}
\item for any $c > 0$ there exists $\Delta_c > 0$ such that $\psi(cx) \le \Delta_c\psi(x)$ for all $x$,
\item $\lim_{c \searrow 0} \Delta_c = 0$.
\end{itemize}
If $\psi$ satisfies the condition ($\Delta_c$), we have $[x]_{\psi\text{-var};I} < \infty$ if and only if $V^1_{\psi;I}(x) < \infty$, and in that case $V^M_{\psi;I}(x) < \infty$ for any $M$.

If $x$ is continuous and $V^1_{\psi;[a,b]}(x) < \infty$, then the function $t \mapsto V^1_{\psi;[a,t]}(x)$ is continuous on $[a,b]$ (cf.\@ \cite[2.14]{LO73}). In particular, it can be parametrised by $\psi$-variation, i.e.\@ $V^1_{\psi;[a,t]}(x) = t-a$. In that parametrisation, it has the following modulus of continuity
\[ d(x(t),x(s)) \le \psi^{-1}(t-s) . \]

We will later consider the following function $\psi$. Let $p \ge 1$, $q \ge 0$. Fix any $x_0 \in {]0,1[}$. We define
\begin{equation}\label{eq:psi_pq}
\psi(x) = \psi_{p,q}(x) = 
\begin{cases}
x^p (\log\frac{1}{x})^{-q} & \text{ for } x \le x_0,\\
\bigg( \psi(x_0)^{1/p}+(\psi^{1/p})'(x_0)(x-x_0) \bigg)^p & \text{ for } x > x_0.
\end{cases}
\end{equation}
The advantage of this choice of $\psi$ is that it is convex on all $\bbR$. Note that $\psi(x) \asymp x^p$ for large $x$. Moreover, one can easily check that
\begin{equation}\label{eq:psi_pq_est}
\psi(xy) \lesssim (xy)^p \left(\logp(1/x)\right)^{-q} (\logp(y))^{q}
\end{equation}
for $x,y \ge 0$.

\subsection{Loewner chains: General driving function}
\label{se:prelim_general}

We briefly summarise the basics on (chordal) Loewner chains and SLE that we will use in this paper. More details can be found e.g.\@ in \cite{law-conformal-book, kem-sle-book}. In this subsection we state the results that hold for any continuous driving function $\xi\colon {[0,\infty[} \to \bbR$. In the next subsection we will focus on Brownian driving functions $\xi(t) = \sqrt{\kappa}B_t$.

We consider the forward (chordal) Loewner differential equation
\begin{equation}\label{eq:loewner}
\partial_t g_t(z) = \frac{2}{g_t(z)-\xi(t)}, \quad g_0(z) = z.
\end{equation} 
The solution $g_t(z)$ exists for $t < T(z)$ where $T(z)$ is the first time where the denominator hits $0$. We write
\begin{align*}
K_t &= \{ z \in \barH \mid T(z) \le t\},\\
H_t &= \{ z \in \bbH \mid T(z) > t\} .
\end{align*}
Then $g_t\colon H_t \to \bbH$ is a conformal map, the so-called mapping-out function of $K_t$. The family $(K_t)$ is parametrised by half-plane capacity, meaning $\hcap(K_t) = 2t$. 

The following estimates on the half-plane capacity will be useful (cf.\@ \cite[Lemma 3.4]{vl-sle-hoelder}, \cite{rw-hcap}):
\begin{equation}\label{eq:hcap_bound}
\operatorname{height}(K)^2 \lesssim \hcap(K) \lesssim \diam(K)\operatorname{height}(K) .
\end{equation} 

We say that the Loewner chain driven by $\xi$ has a continuous trace if $\gamma(t) = \linebreak \lim_{v \searrow 0} g_t^{-1}(\xi(t)+iv)$ exists and is a continuous function in $t$. This is equivalent to saying that there exists a continuous $\gamma\colon {[0,\infty[} \to \barH$ such that for each $t \ge 0$ the domain $H_t$ is the unbounded connected component of $\bbH \setminus \gamma[0,t]$. The existence of a trace has been shown for a wide class of driving functions \cite{lind-le-slit, stw-finite-var}, and a.s.\@ for Brownian motion with speed $\kappa$ which gives us \slek{} \cite{rs-sle, lsw-lerw-ust}. (We will give in \cref{se:pf_warmup} another proof in the case $\kappa\neq 8$.)

We write $Z_t(z) = X_t(z)+iY_t(z) = g_t(z)-\xi(t)$. Sometimes, to ease the notation, we will leave out the parameter $z$ when there is no confusion. The equation for $g_t$ rewrites to
\begin{align}
\begin{split}\label{eq:loewner_xy}
dX_t &= \frac{2X_t}{X_t^2+Y_t^2} \, dt -d\xi(t) ,\\
dY_t &= \frac{-2Y_t}{X_t^2+Y_t^2} \, dt.
\end{split}
\end{align}
Then (cf.\@ \cite{rs-sle})
\begin{equation*}
    |g_t'(z)| = \exp \left( -2 \int_0^t \frac{X_s^2-Y_s^2}{(X_s^2+Y_s^2)^2} \, ds \right).
\end{equation*}
Moreover, we write $\hat f_t(w) = g_t^{-1}(w+\xi(t))$.

Some immediate consequences of \eqref{eq:loewner_xy} are: \\
(i) $\Re g_t(z)$ is strictly increasing (resp.\@ decreasing) and $X_t(z) > 0$ (resp.\@ $X_t(z) < 0$) when $\Re z > \sup_t \xi(t)$ (resp.\@ $\Re z < \inf_t \xi(t)$).\\
(ii) We always have $\partial_t Y_t^2 \in [-4,0[$. Moreover, by the Schwarz lemma we always have $\abs{\hat f'_t(w)} \le \frac{\Im \hat f_t(w)}{\Im w} \le \frac{\sqrt{v^2+4t}}{v}$ where $v=\Im w$.

For $z \in \bbH$ we let $\Upsilon_t(z) = \frac{Y_t(z)}{\abs{g_t'(z)}} = \frac{1}{2}\crad(z,H_t)$ where $\crad$ denotes conformal radius. We have
\[ \partial_t \Upsilon_t = \frac{-4Y_t^2}{(X_t^2+Y_t^2)^2}\Upsilon_t . \]
The parametrisation by conformal radius (introduced in \cite{lw-sle-greens}) is defined via
\[ \sigma(s) = \sigma(s,z) = \inf\{t \ge 0 \mid \Upsilon_t(z) = e^{-4s}\} .\]
Notice that the $s$-parametrisation starts at $s_0(y) \defeq -\frac{1}{4}\log y$, i.e.\@ $\sigma(s_0,z)=0$.
We have the identities
\begin{align}
d\sigma(s) &= Y_{\sigma(s)}^2 \left( 1+\frac{X_{\sigma(s)}^2}{Y_{\sigma(s)}^2}\right)^2 \,ds , \label{eq:sigma_dynamics}\\
\partial_s Y_{\sigma(s)}^2 &= -4(X_{\sigma(s)}^2+Y_{\sigma(s)}^2) = -4Y_{\sigma(s)}^2\left( 1+\frac{X_{\sigma(s)}^2}{Y_{\sigma(s)}^2}\right) , \label{eq:Ysigma_dynamics}\\
d\frac{X_{\sigma(s)}}{Y_{\sigma(s)}} &= 4\frac{X_{\sigma(s)}}{Y_{\sigma(s)}} \left( 1+\frac{X_{\sigma(s)}^2}{Y_{\sigma(s)}^2}\right) \,ds - \frac{1}{Y_{\sigma(s)}} \,d\xi(\sigma(s)) .
\end{align}
Let $\hat\theta(s) = \cot^{-1}\left(\frac{X_{\sigma(s)}}{Y_{\sigma(s)}}\right) \in {]0,\pi[}$. Writing everything in terms of $\hat\theta$ will be convenient in case $\xi(t) = \sqrt{\kappa}B_t$ because $\hat\theta$ will be a radial Bessel process (see \cref{se:prelim_bm}).
We then have
\begin{align*}
Y_{\sigma(s)} &= Y_0 \exp\left( -2 \int_{s_0}^s (\sin\hat\theta_r)^{-2} \, dr \right) \\
&= Y_0 \exp\left( -2(s-s_0)-2 \int_{s_0}^s \cot^2\hat\theta_r \, dr \right) , \\
\abs{g_{\sigma(s)}'(z)} &= \exp\left( 4(s-s_0)-2 \int_{s_0}^s (\sin\hat\theta_r)^{-2} \, dr \right) \\
&= \exp\left( 2(s-s_0)-2 \int_{s_0}^s \cot^2\hat\theta_r \, dr \right) .
\end{align*}

We will frequently use the following estimates. Let $s,t \ge 0$ and $v > 0$. From Koebe's distortion theorem (cf.\@ \cite[Theorem 1.6]{pom-univalent-book}), we see that
\begin{equation}\label{eq:f_int}
\abs{\hat f_t(i2v)-\hat f_t(iv)} \le \int_v^{2v} \abs{\hat f_t'(iu)}\,du \asymp v\abs{\hat f_t'(iv)} = \Upsilon_t(\hat f_t(iv)) .
\end{equation}
Moreover, if $\abs{t-s} \asymp v^2$, by \cite[Lemma 4.5]{fty-regularity-kappa} (which is a restatement of \cite[Lemma 3.5 and 3.2]{vl-sle-hoelder}) we have
\begin{equation} \begin{split}\label{eq:f_diff}
\abs{\hat f_t(iv)-\hat f_s(iv)} 
&\lesssim \abs{\hat f_s'(iv)} \left( \frac{\abs{t-s}}{v} + \abs{\xi(t)-\xi(s)}\left(1+\frac{\abs{\xi(t)-\xi(s)}^2}{v^2}\right)^l \right) \\
&\asymp \Upsilon_s(\hat f_s(iv)) \left( 1 + \left(\frac{\abs{\xi(t)-\xi(s)}^2}{\abs{t-s}} \right)^l \right)
\end{split} \end{equation}
where $l < \infty$ is a universal constant.

\begin{remark}\label{rm:crad_vs_path}
The quantity $\Upsilon_s(\hat f_s(iv))$ plays a major role in estimating the increment $\abs{\gamma(t)-\gamma(s)}$. Indeed, by \cite[Proposition 3.10]{vl-sle-hoelder} the increment is (more or less) lower bounded by $\Upsilon_s(\hat f_s(iv))$ where $v = \abs{t-s}^{1/2}$. For an upper bound, a commonly used estimate is
\[ \abs{\gamma(t)-\gamma(s)} \le \abs{\gamma(t)-\hat f_t(iv)}+\abs{\hat f_t(iv)-\hat f_s(iv)}+\abs{\hat f_s(iv)-\gamma(s)} , \]
and then by \eqref{eq:f_int}
\[ \abs{\gamma(s)-\hat f_s(iv)} \lesssim \sum_m \Upsilon_s(\hat f_s(iv2^{-m}))
 . \]
In case $\Upsilon_s(\hat f_s(iv)) \approx v^{1-\beta}$ for some $\beta < 1$, the upper and lower bound match (up to a factor). Therefore this strategy gives us sharp $p$-variation and Hölder exponents in the case $\kappa\neq 8$, as was proved in \cite{vl-sle-hoelder,ft-sle-regularity}.

For $\kappa=8$, we only have $\sup_s \Upsilon_s(\hat f_s(iv)) \approx (\log 1/v)^{-1/4}$, and the upper bound is too bad to obtain any regularity results. In that case, an improvement is offered by \cite[Lemma 7.1]{kms-sle48x}, see \cref{se:sle8} for more details.
\end{remark}

To estimate the regularity of the trace $\gamma$, we follow the following strategy which has also been used in \cite{msy-le-semimartingalex}. For the reader's convenience, we also restate the proofs of the lemmas below.

As explained above, we want to find an upper bound for $\abs{\hat f_t'(iv)} = \abs{g_t'(\hat f_t(iv))}^{-1}$ where $v > 0$. Observe that $z = \hat f_t(iv)$ is the point where we have to start the flow in order to reach $Z_t = iv$. But since this point depends on the behaviour of $\xi$ in the future time interval $[0,t]$, we would need to consider all possible points $z \in \bbH$ that might reach $iv$ at time $t$. Fortunately, using Koebe's distortion and $1/4$-theorem, we can reduce the problem to starting the flow from a finite number of points. Then the number of points we need to test already encodes information on $\abs{\hat f_t'(iv)}$.

Recall Koebe's distortion estimates and a few consequences.
\begin{lemma}\label{le:distortion}
Let $f\colon \bbH \to \bbC$ be a univalent function, $g = f^{-1}\colon f(\bbH) \to \bbH$, and $w = u+iv \in \bbH$. Then for every $z \in B(f(w),\frac{1}{8}v\abs{f'(w)})$ we have
\[ \abs{g(z)-w} < \frac{v}{2} \quad\text{and}\quad \frac{48}{125} \le \frac{\abs{g'(z)}}{\abs{g'(f(w))}} \le \frac{80}{27} . \]
\end{lemma}

\begin{proof}
Note that $v\abs{f'(w)} = \frac{1}{2}\crad(f(w),f(\bbH))$. In particular, from Koebe's $1/4$ theorem we know that $\dist(f(w),\partial f(\bbH)) \ge \frac{1}{4}\crad(f(w),f(\bbH)) = \frac{1}{2}v\abs{f'(w)}$. Another application of Koebe's $1/4$ theorem implies $f(B(w,v/2)) \supseteq B(f(w),\frac{1}{8}v\abs{f'(w)})$.

In particular, every $z \in B(f(w),\frac{1}{8}v\abs{f'(w)})$ satisfies $\abs{f^{-1}(z)-w} < v/2$. We conclude by Koebe's distortion theorem applied on the domain $B(f(w),\frac{1}{2}v\abs{f'(w)})$.
\end{proof}

This motivates to start the Loewner flow from the following set of points
\begin{equation}\label{eq:fw_grid}
H(h,M,T) = \left\{ x+iy \mmiddle| 
\begin{array}{cc} x = \pm hj/8,\ y = h(1+k/8),\ j,k \in \bbN \cup \{0\},\\
\abs{x} \le M,\ y \le \sqrt{1+4T} 
\end{array} \right\} .
\end{equation}
This grid is chosen in such a way that for every $z \in [-M,M] \times [h,\sqrt{1+4T}]$ we have $\dist(z,H(h,M,T)) < \frac{h}{8}$.

The following lemma is purely deterministic and holds for any continuous driving function $\xi$.

\begin{lemma}\label{thm:fw_grid}
Let $v \in {]0,1]}$, $r > 0$, and suppose $|\hat f_t'(iv)| \ge r$ for some $t \in [0,T]$. Then there exists $z \in H(rv,\norm{\xi}_{\infty;[0,T]},T)$ such that
\[ \quad |Z_{t}(z)-iv| \le \frac{v}{2} \quad \text{and} \quad |g_{t}'(z)| \le \frac{80}{27}\,\frac{1}{r} \]
where $H(h,M,T)$ is given by \eqref{eq:fw_grid}.
\end{lemma}

\begin{remark}
For later reference, let us note here that the condition $|Z_{t}(z)-iv| \le v/2$ implies in particular
\[
Y_t(z) \in [v/2, 3v/2] \quad \text{and} \quad \abs*{\frac{X_t(z)}{Y_t(z)}} \le 1 .
\]
\end{remark}

\begin{proof}
Surely, there is $z_*=\hat f_t(iv)$ which by definition satisfies everything, but the claim is that we can choose $z$ from the grid $H(rv,M,T)$. Indeed, the grid is just defined so that there always exists some $z \in H(rv,M,T)$ with $\abs{z-z_*} \le \frac{1}{8}rv$ provided that $z_* \in [-M,M] \times [rv,\sqrt{1+4T}]$. By \cref{le:distortion}, such $z$ satisfies the desired properties.

The fact that $z_* \in [-M,M] \times [rv,\sqrt{1+4T}]$ just comes from the Loewner equation (for the upper bounds) and from the Schwarz lemma (for the lower bound); see the comments below \eqref{eq:loewner_xy}.
\end{proof}

We will use in \cref{se:unif_map} a variant of \cref{thm:fw_grid}.

\begin{lemma}\label{le:fw_grid_all}
Let $v \in {]0,1]}$, $r > 12$, and suppose $|\hat f_t'(u+iv)| \ge r$ for some $u \in \bbR$, $t \in [0,T]$. Then there exists $z \in H(rv, 2\norm{\xi}_{\infty;[0,T]}+4\sqrt{T}, T)$ such that
\[ \quad |Z_{t}(z)-(u+iv)| \le \frac{v}{2} \quad \text{and} \quad |g_{t}'(z)| \le \frac{80}{27}\,\frac{1}{r} \]
where $H(h,M,T)$ is given by \eqref{eq:fw_grid}.
\end{lemma}

\begin{proof}
The proof is the same except that this time we need to justify $\hat f_t(u+iv) \in [\pm (2\norm{\xi}_{\infty;[0,T]}+4\sqrt{T})] \times [rv,\sqrt{1+4T}]$. The Loewner differential equation implies $K_T \subseteq [\pm\norm{\xi}_{\infty;[0,T]}] \times [0,2\sqrt{T}]$ (cf.\@ the comment below \eqref{eq:loewner_xy}). Applying Koebe's distortion theorem to the map $z \mapsto 1/g_t(1/z)$ we see that if $\abs{z} > 2\norm{\xi}_{\infty;[0,T]}+4\sqrt{T}$, then
\[ \frac{1}{12} \le \abs{g_t'(z)} \le \frac{27}{4} . \]
These points are excluded by the condition $r > 12$.
\end{proof}

We will later need to sum up certain expressions on the grid $H(h,M,T)$. We state here the calculation that we will use later.
\begin{lemma}\label{le:fw_sum_grid}
Let $a,\zeta \in \bbR$, $M,T > 0$, $h \in {]0,1]}$. Then there exists $C < \infty$ depending on $a,\zeta,M,T$ such that
\[ \sum_{z \in H(h,M,T)} y^\zeta(1+x^2/y^2)^{-a/2} \le C
\begin{cases}
h^\zeta & \text{if } a > 1,\ \zeta+1 < -1,\\
h^{-2}\logp(h^{-1}) & \text{if } a > 1,\ \zeta+1 = -1,\\
h^{-2} & \text{if } a > 1,\ \zeta+1 > -1,\\[5pt]

h^\zeta \logp(h^{-1}) & \text{if } a = 1,\ \zeta+1 < -1,\\
h^{-2}\logp(h^{-1})^2 & \text{if } a = 1,\ \zeta+1 = -1,\\
h^{-2} & \text{if } a = 1,\ \zeta+1 > -1,\\[5pt]

h^{\zeta+a-1} & \text{if } a < 1,\ \zeta+a < -1,\\
h^{-2}\logp(h^{-1}) & \text{if } a < 1,\ \zeta+a = -1,\\
h^{-2} & \text{if } a < 1,\ \zeta+a > -1.
\end{cases} \]
\end{lemma}

\begin{remark}
The constant $C$ depends on $M,T$ polynomially.
\end{remark}

\begin{remark}\label{rm:sum_grid_capped}
Suppose we only sum over $z$ with $y=\Im z \ge \varepsilon$. In case $a>1$, $\zeta+2<0$, we then get $\sum \le C\varepsilon^{\zeta+2}h^{-2}$. Analogous statements hold in the other cases.
\end{remark}

\begin{proof}
For simplicity, we can write $x_j = hj$, $y_k = hk$ where $j = -\lceil M h^{-1} \rceil,...,\lceil M h^{-1} \rceil$ and $k = 1,...,\lceil\sqrt{1+T}h^{-1}\rceil$. (The additional factors do not matter and will be absorbed in the final constant $C$.)

We have
\[ y_k^\zeta(1+x_j^2/y_k^2)^{-a/2} = (hk)^\zeta (1+j^2/k^2)^{-a/2} . \]
We first sum in $j$.
\begin{align*}
    \sum_{j \le Mh^{-1}} (1+j^2/k^2)^{-a/2} 
    &\asymp \int_0^{Mh^{-1}} (1+j^2/k^2)^{-a/2} \, dj \\
    &= \int_0^{Mh^{-1}/k} (1+j'^2)^{-a/2}k \, dj' \\
    &\asymp \begin{cases}
    k & \text{if } a>1,\\
    k\logp(\frac{M}{hk}) & \text{if } a=1,\\
    h^{a-1}k^{a} & \text{if } a<1.
    \end{cases}
\end{align*} 
We then sum in $k$. In case $a>1$ we have
\[
    \sum_{k=1}^{\sqrt{1+T}h^{-1}} (hk)^\zeta k 
    \asymp \begin{cases}
    h^\zeta & \text{if } \zeta+1 < -1,\\
    h^{-2}\logp(\sqrt{1+T}h^{-1}) & \text{if } \zeta+1 = -1,\\
    h^{-2} & \text{if } \zeta+1 > -1.
    \end{cases}
\]
In case $a<1$ we have
\[
    \sum_{k=1}^{\sqrt{1+T}h^{-1}} (hk)^\zeta h^{a-1}k^{a} 
    \asymp \begin{cases}
    h^{\zeta+a-1} & \text{if } \zeta+a < -1,\\
    h^{-2}\logp(\sqrt{1+T}h^{-1}) & \text{if } \zeta+a = -1,\\
    h^{-2} & \text{if } \zeta+a > -1.
    \end{cases}
\]
In case $a=1$ we have
\[
    \sum_{k=1}^{\sqrt{1+T}h^{-1}} (hk)^\zeta k\logp(\frac{M}{hk}) 
    \asymp \begin{cases}
    h^\zeta \logp(\frac{M}{h}) & \text{if } \zeta+1 < -1,\\
    h^{-2}\logp(\frac{M}{h})\logp(\frac{\sqrt{1+T}}{h}) & \text{if } \zeta+1 = -1,\\
    h^{-2} & \text{if } \zeta+1 > -1,
    \end{cases}
\]
\end{proof}

\subsection{Loewner chains: Brownian driving function}
\label{se:prelim_bm}

Suppose in the following that $\xi(t) = \sqrt{\kappa}B_t$ where $B$ is a standard Brownian motion and $\kappa \ge 0$. We denote the filtration generated by $B$ by $\mathcal F = (\mathcal F_t)$.

With $Z_t = g_t(z)-\sqrt{\kappa}B_t$, the equation \eqref{eq:loewner} can be rewritten as
\[ dZ_t = \frac{2}{Z_t}\,dt-\sqrt{\kappa}\,dB_t \]
which can be seen as a complex version of the (usual) Bessel process. In particular, for initial values $z=x \in \bbR$, the process $Z_t(x)=X_t(x)$ is a real Bessel process of index $\tilde\nu = \frac{2}{\kappa}-\frac{1}{2}$ (equivalently dimension $1+\frac{4}{\kappa}$) run with speed $\kappa$.

Recall that a Bessel process $(\rho_t)$ of positive index $\tilde\nu > 0$ is transient and satisfies
\begin{equation*}
\bbP( \rho_t \le \varepsilon \text{ for some } t \ge 1 ) \asymp \varepsilon^{2\tilde\nu} . 
\end{equation*}
The latter can be derived from the following two facts:\\
1. The transition probability of the Bessel process is (cf.\@ \cite[p.\@ 446]{ry-stochastic-book})
\[ p_t(0,y) = c\,t^{-(\tilde\nu+1)}y^{2\tilde\nu+1}\exp(-y^2/2t) . \]
2. The hitting time of the Bessel process satisfies (cf.\@ \cite[p.\@ 442]{ry-stochastic-book})
\[ \bbP_x( T_\varepsilon < \infty ) = \left(\frac{\varepsilon}{x}\right)^{2\tilde\nu} \quad\text{for } x > \varepsilon . \]

From Brownian scaling, it follows that
\begin{equation}\label{eq:bessel_hit_after_t}
\bbP( \rho_t \le \varepsilon \text{ for some } t \ge t_0 ) \asymp t_0^{-\tilde\nu}\varepsilon^{2\tilde\nu} . 
\end{equation}

The following lemma is useful for removing the boundary effect in the statement of \cref{thm:main_hoelder}.

\begin{lemma}\label{le:sle_hit_after_1}
Let $\kappa < 4$. For $\varepsilon > 0$ we have with probability $1-O(\varepsilon^{\frac{4}{\kappa}-1})$ that
\[
\abs{X_t(z)} \ge \varepsilon \quad \text{for all } z \in H_t \cap \{ \Im z \le \varepsilon \} , \ t \ge 1 .
\]
\end{lemma}

\begin{proof}
It suffices to show this for small $\varepsilon$.

We make use of a few known results about SLE. By, \cite[(1.4)]{sz-boundary-proximity}, the \slek{} trace does not intersect the set $({]-\infty,-\varepsilon^{1/2}]}\cup{[\varepsilon^{1/2},\infty[}) \times [0,\varepsilon]$ with probability $1-O(\varepsilon^{\frac{4}{\kappa}-1})$. The probability of the trace intersecting the set $[\pm\varepsilon^{1/2}] \times [0,\varepsilon]$ after time $1$ will be estimated using \cite[Theorem 1.1]{fl-sle-transience}.

Note that the Loewner equation implies $\sup_{t \in [0,t_0]} \abs{\Re\gamma(t)} \le \sup_{t \in [0,t_0]} \abs{\xi(t)}$. Therefore \eqref{eq:hcap_bound} implies
\[ \begin{split}
\bbP\left( \Im\gamma(t) \ge (\log 1/\varepsilon)^{-1} \text{ for some } t \in [0,1/2] \right) 
&\ge \bbP\left( \sup_{[0,1/2]} \abs{\xi(t)} \lesssim \log 1/\varepsilon \right) \\
&\ge 1-\exp(-c(\log 1/\varepsilon)^2) . 
\end{split} \]
Suppose that $\tau = \inf\{ t \mid \Im\gamma(t) \ge (\log 1/\varepsilon)^{-1} \} \le 1/2$. Restarting the SLE flow at time $\tau$, by \eqref{eq:bessel_hit_after_t}, we then have for fixed $c_1 > 0$
\[ 
\bbP\left(\begin{array}{cc}
\Im\gamma(t) \ge (\log 1/\varepsilon)^{-1} \text{ for some } t \in [0,1/2], \\
\abs{X_t(z)} \ge c_1\varepsilon \text{ for all } z \in \gamma[0,\tau] \cup \bbR,\ t \ge 1
\end{array}\right) 
\ge 1-c\varepsilon^{2\tilde\nu} 
\]
where $X_t(z) = g_t(z)-\xi(t)$ and $g_t(z)$ is understood as the continuous extension of $g_t$ to the boundary. (Indeed, since $g_t$ preserves the ordering of the boundary points, it suffices to consider the boundary points infinitesimally to the left and right of $\gamma(\tau)$. They evolve as a Bessel process after time $\tau$.)

Suppose that this event happens, and suppose additionally that $\gamma$ does not re-enter the set $\{ \Im z \le \varepsilon \}$ after time $\tau$. In that case, it follows from \cite[Proposition 3.82]{law-conformal-book} that $\abs{X_t(z)} \ge c_1\varepsilon-c_2\varepsilon$ for all $z \in H_1 \cap \{ \Im z \le \varepsilon \}$ where $c_2$ is a universal constant.

To finish the proof, we need to bound the probability of $\gamma$ re-entering the set $\{ \Im z \le \varepsilon \}$. As remarked above, it remains to bound the probability of re-entering the set $[\pm\varepsilon^{1/2}] \times [0,\varepsilon]$ after time $\tau$. By \cite[Theorem 1.1]{fl-sle-transience}, this probability is bounded by (a constant times) $(\varepsilon^{1/2}(\log 1/\varepsilon))^{8/\kappa-1}$.
\end{proof}

Now we consider the parametrisation by conformal radius introduced in \cref{se:prelim_general}. Recall that we have the representation
\[ d\frac{X_{\sigma(s)}}{Y_{\sigma(s)}} = 4\frac{X_{\sigma(s)}}{Y_{\sigma(s)}} \left( 1+\frac{X_{\sigma(s)}^2}{Y_{\sigma(s)}^2}\right) \,ds - \left( 1+\frac{X_{\sigma(s)}^2}{Y_{\sigma(s)}^2}\right) \sqrt{\kappa}\,d\hat B_s \]
where $d\hat B_s = \frac{1}{Y_{\sigma(s)}(1+X_{\sigma(s)}^2/Y_{\sigma(s)}^2)} \,dB_{\sigma(s)}$ is another standard Brownian motion.
For the process $\hat\theta(s) = \cot^{-1}\left(\frac{X_{\sigma(s)}}{Y_{\sigma(s)}}\right) \in {]0,\pi[}$ we have
\[ d\hat\theta_s = (\kappa-4)\cot\hat\theta_s\,ds+\sqrt{\kappa}\,d\hat B_s . \]
This is a (time-changed) radial Bessel process of index $\nu = \frac{1}{2}-\frac{4}{\kappa}$ (i.e.\@ dimension $3-\frac{8}{\kappa}$). In particular, it hits the boundary $\{0,\pi\}$ in finite time if and only if $\kappa < 8$. (This reflects the fact that for $\kappa < 8$, each point $z \in \bbH$ is a.s.\@ missed or swallowed in finite time, whereas for $\kappa \ge 8$, each point $z \in \bbH$ is a.s.\@ hit in finite time.)

We assume absorbing boundary for the process, i.e.\@ we stop the process when $\hat\theta_\tau \in \{0,\pi\}$. In case $\kappa \in [0,4]$, this happens only when $\sigma(\tau) = \infty$. In case $\kappa \in {]4,8[}$, this happens when $z$ is swallowed (cf.\@ \cite[Lemma 3]{sch-percolation-formula}).

\begin{remark}
It turns out that under the radial parametrisation the ``one-point martingale'' for SLE (cf.\@ \cite{rs-sle, lw-sle-greens})
\[ M_{\sigma(s)} = \abs{g_{\sigma(s)}'(z)}^\lambda Y_{\sigma(s)}^\zeta (1+X_{\sigma(s)}^2/Y_{\sigma(s)}^2)^{r/2} 
= e^{4s\lambda} Y_{\sigma(s)}^{\zeta+\lambda} (1+X_{\sigma(s)}^2/Y_{\sigma(s)}^2)^{r/2} . \]
is exactly the ``change of measure''-martingale for radial Bessel processes, see \cref{se:bessel}. (This is consistent with the fact that it is the density of SLE$_{\kappa}(\rho)$ with force point at $z$ where $\rho = r\kappa$ (cf.\@ \cite{sw-sle-coordinate-change, zhan-decomposition}).)
\end{remark}

\subsection{Radial Bessel process}
\label{se:bessel}

Part of the material presented here are contained in \cite{law-bessel-notes}.

Consider a radial Bessel process of index $\nu$ (equivalently dimension $\delta=2+2\nu$)
\[ d\theta_t = (\frac{1}{2}+\nu)\cot\theta_t \, dt + dB_t , \quad \theta_0 \in {]0,\pi[} . \]
Recall that such process hits $\{0,\pi\}$ in finite time if and only if $\nu < 0$.

We introduce the ``change of measure''-martingale for the Bessel process: $M_t = (\frac{\sin\theta_t}{\sin\theta_0})^a K_t$ where $K_t = K^{(\nu,a)}_t$ is a compensator that we compute now. We have
\[ d(\sin\theta_t)^a = (\sin\theta_t)^a \left( a\cot\theta_t \,dB_t + \left[ -\frac{a}{2}+a(\frac{a}{2}+\nu)\cot^2\theta_t \right] \, dt \right) . \]
This leads us to
\[ \begin{split}
K_t &= \exp\left( \int_0^t \left[ \frac{a}{2}-a(\frac{a}{2}+\nu)\cot^2\theta_s \right] \, ds \right) \\
&= \exp\left( a\left[\frac{a}{2}+\frac{1}{2}+\nu\right] t -a(\frac{a}{2}+\nu) \int_0^t (\sin\theta_s)^{-2} \, ds \right) .
\end{split} \]
Then
\[ M_t = \left(\frac{\sin\theta_t}{\sin\theta_0}\right)^a K_t \]
is a local martingale with
\[ dM_t = a\cot\theta_t\, M_t \, dB_t . \]
It is a bounded martingale until time $t \wedge T_\varepsilon$ where $T_\varepsilon \defeq \inf\{ t \mid \theta_t \in \{\varepsilon,\pi-\varepsilon\}\}$.

Applying Girsanov's theorem to $M_{t \wedge T_\varepsilon}$ we get a measure $\bbP^{\nu+a,\varepsilon}$ defined by $d\bbP^{\nu+a,\varepsilon} = M_{t \wedge T_\varepsilon} \, d\bbP^\nu$. We can then write
\[ d\theta_t = (\frac{1}{2}+\nu+a)\cot\theta_t \, dt + d\hat B_t , \quad t \le T_\varepsilon \]
where $\hat B = \hat B^{(\nu,a)}$ is a standard Brownian motion under the measure $\bbP^{\nu+a,\varepsilon}$ until time $T_\varepsilon$.

We claim that in case $\nu+a \ge 0$, $M_t$ is a martingale. Indeed, it suffices to show that
\[ \bbE^\nu[M_{t \wedge T_\varepsilon} 1_{T_\varepsilon < t}] = \bbP^{\nu+a,\varepsilon}(T_\varepsilon < t) \to 0 \quad \text{as } \varepsilon \searrow 0 \]
since the optional sampling theorem (applied to $M_{t \wedge T_\varepsilon}$) and the monotone convergence theorem will then imply $\bbE^\nu M_t = \bbE^\nu M_0$. But this is true because the index of the Bessel process $\theta$ under the law $\bbP^{\nu+a,\varepsilon}$ is $\nu+a \ge 0$.

In particular, $d\bbP^{\nu+a} = M_t\,d\bbP^\nu$ defines a probability measure in case $\nu+a \ge 0$.

In case $\nu+a < 0$, for every $t$ the measure $\bbP^{\nu+a}$ is still defined on $\mathcal F_t\big|_{\{T_0 > t\}}$, and $\frac{d\bbP^{\nu+a}}{d\bbP^\nu}\big|_{\{T_0 > t\}} = M_t$.

Let $q_t(x,y)$ denote the transition density of the process $\theta$ under $\bbP^\nu$. It can be written down explicitly, see \cite[Proposition 8.1]{zhan-ergodicity}, and satisfies the following bound. In particular, it converges exponentially fast to its stationary law $f(y) = c_\nu(\sin y)^{1+2\nu}$.

\begin{proposition}\label{prop:bessel_density}
If $\nu \ge 0$, then there exist $c < \infty$ such that for all $t \ge 1$ and $x,y \in {]0,\pi[}$
\[ (1-ce^{-(1+\nu) t}) f(y) \le q_t(x,y) \le (1+ce^{-(1+\nu) t}) f(y) \]
where $f(y) = c_\nu(\sin y)^{1+2\nu}$.
\end{proposition}

We are interested in the ``radial Bessel clock''
\[ C_t = \int_0^t (\sin\theta_s)^{-2} \, ds . \]

In case $\nu > 0$, we can pick $a \in {]-2\nu,0[}$, and the fact that $M_t$ is a martingale implies that $C_t$ has exponential moments of order $\frac{\nu^2}{2}$.

\begin{proposition}\label{prop:bessel_clock_moments_nug0}
If $\nu > 0$ and $p>0$, then there exists $c < \infty$ (depending on $\nu,p$) such that
\[ \bbE^\nu[C_t^p] \le c(1+(-\log\sin\theta_0)^p+t^p) \]
for $t \ge 1$.
\end{proposition}

\begin{proof}
We have
\[ \log\sin\theta_t = \log\sin\theta_0 + \int_0^t \cot\theta_s \, dB_s - (\frac{1}{2}+\nu)t + \nu \int_0^t (\sin\theta_s)^{-2} \, ds , \]
or equivalently
\[ \nu C_t = \log\sin\theta_t-\log\sin\theta_0 - \int_0^t \cot\theta_s \, dB_s + (\frac{1}{2}+\nu)t . \]
Note that
\[ \left\langle \int_0^t \cot\theta_s \, dB_s \right\rangle = \int_0^t \cot^2\theta_s \, ds = C_t-t \]
and therefore (by the BDG inequality)
\[ \begin{split}
\bbE C_t^p &\lesssim (-\log\sin\theta_0)^p + \bbE(-\log\sin\theta_t)^p + t^p + \bbE(C_t)^{p/2} \\
&\le (-\log\sin\theta_0)^p + c + t^p + (\bbE C_t^p)^{1/2} 
\end{split} \]
where we have used \cref{prop:bessel_density} to bound $\bbE(-\log\sin\theta_t)^p$ by a constant $c$ independent of $t \ge 1$.

Recall that (in the case $\nu > 0$) $C_t$ has exponential moments of small order. Hence $\bbE C_t^p < \infty$, and solving the quadratic equation yields
\[ (\bbE C_t^p)^{1/2} \lesssim c+\sqrt{(-\log\sin\theta_0)^p + c + t^p} \]
for all $t$.
\end{proof}

\begin{proposition}\label{prop:bessel_clock_mixed_moments}
Let $\nu \in \bbR$ and $\lambda < \frac{\nu^2}{2}$. Let $\tau$ be a bounded stopping time and $X$ be a $\mathcal F_\tau$-measurable random variable. Then
\[ \bbE^\nu[ X \exp(\lambda C_\tau) \,1_{T_0 > \tau}] = (\sin\theta_0)^{a} \,\bbE^{\nu+a}\left[X (\sin\theta_\tau)^{-a} \exp\left((\lambda-\frac{a}{2})\tau\right)\right] \]
where $a = -\nu+\sqrt{\nu^2-2\lambda}$.\\
Moreover, suppose additionally that either $\nu > -2$ or (in the case $\nu \le -2$) $\lambda < -2\nu-2$.
Then, for $p>0$, there exists $c < \infty$ (depending on $\nu,\lambda,p$) such that
\[ \bbE^\nu[ \exp(\lambda C_t) C_t^p \,1_{T_0 > t}] \le c (\sin\theta_0)^{a} (1+(-\log\sin\theta_0)^p+t^p) \exp\left((\lambda-\frac{a}{2})t\right) \]
for $t \ge 1$.
\end{proposition}

\begin{proof}
The parameter $a$ is chosen such that $\lambda = -a(\frac{a}{2}+\nu)$ and $\nu+a > 0$. In that case,
\[ M_t = \left(\frac{\sin\theta_t}{\sin\theta_0}\right)^a \exp\left( (\frac{a}{2}-\lambda) t + \lambda C_t \right) \]
is a $\bbP^\nu$-martingale, and the law of $\theta$ under $\bbP^{\nu+a}$ is that of a radial Bessel process of index $\nu+a > 0$. Therefore
\[ \begin{split}
\bbE^\nu[ X \exp(\lambda C_\tau) \,1_{T_0 > \tau}] 
&= \bbE^{\nu+a}[ X \exp(\lambda C_\tau) M_\tau^{-1} \,1_{T_0 > \tau}] \\
&= (\sin\theta_0)^{a} \,\bbE^{\nu+a}\left[X (\sin\theta_\tau)^{-a} \exp\left((\lambda-\frac{a}{2})\tau\right)\right] . 
\end{split} \]
To get the second claim, we apply this to $X = C_t^p$, then use Hölder's inequality, \cref{prop:bessel_clock_moments_nug0}, and \cref{prop:bessel_density} to obtain
\[ \begin{split}
\bbE^{\nu+a}[C_t^p (\sin\theta_t)^{-a}] 
&\le \left( \bbE^{\nu+a}[C_t^{pq'}] \right)^{1/q'} \left( \bbE^{\nu+a}[(\sin\theta_t)^{-aq}] \right)^{1/q}\\
&\lesssim 1+(-\log\sin\theta_0)^p+t^p
\end{split} \]
provided that the second expectation is finite and bounded for some choice of $q>1$. By \cref{prop:bessel_density} this is the case whenever $-aq+1+2(\nu+a) > -1$, and we can pick such $q>1$ if and only if $\nu > -2$ or $\lambda < -2\nu-2$.
\end{proof}

In the special case $\lambda = 0$, we get the following statement.

\begin{corollary}\label{prop:bessel_clock_moments_nul0}
Let $\nu < 0$. Let $\tau$ be a bounded stopping time and $X$ be a $\mathcal F_\tau$-measurable random variable. Then
\[ \bbE^\nu[X \,1_{T_0 > \tau}] = (\sin\theta_0)^{-2\nu} \,\bbE^{-\nu}[X (\sin\theta_\tau)^{2\nu} e^{\nu \tau}] . \]
Moreover, for $p>0$, there exists $c < \infty$ (depending on $\nu,p$) such that
\[ \bbE^\nu[C_t^p \,1_{T_0 > t}] \le c (\sin\theta_0)^{-2\nu} (1+(-\log\sin\theta_0)^p+t^p) e^{\nu t} \]
for $t \ge 1$.
\end{corollary}

We have a similar statement in case $p<0$.

\begin{corollary}\label{prop:bessel_clock_mixed_moments_neg}
Let $\nu \in \bbR$ and $\lambda \in {]0,\frac{\nu^2}{2}[}$ such that $\nu > -2$ or $\lambda < -2\nu-2$. Then, for $p>0$, there exists $c<\infty$ (depending on $\nu,\lambda,p$) such that
\[ \bbE^\nu[ \exp(\lambda C_t) (1+C_t)^{-p} \,1_{T_0 > t}] \le c(\sin\theta_0)^{a} \exp\left((\lambda-\frac{a}{2})t\right)t^{-p} \]
for $t \ge 1$ where $a = -\nu+\sqrt{\nu^2-2\lambda}$.
\end{corollary}

\begin{proof}
We split up into the events $\{ C_t \le \delta t \}$ and $\{ C_t \ge \delta t \}$ where $\delta > 0$ is a suitably chosen number. On the event $\{ C_t \ge \delta t \}$ we have
\[ \begin{split}
&\bbE^\nu[ \exp(\lambda C_t) (1+C_t)^{-p} \,1_{T_0 > t} 1_{C_t \ge \delta t}] \\
&\quad \le (1+\delta t)^{-p}\, \bbE^\nu[ \exp(\lambda C_t) \,1_{T_0 > t}] \\
&\quad \le (1+\delta t)^{-p} (\sin\theta_0)^{a} \exp\left((\lambda-\frac{a}{2})t\right) \,\bbE^{\nu+a}\left[(\sin\theta_t)^{-a} \right] ,
\end{split} \]
and the last expectation is bounded as in the proof of \cref{prop:bessel_clock_mixed_moments}.

For the event $\{ C_t \le \delta t \}$ we distinguish the cases $\nu \ge 0$ and $\nu < 0$. In case $\nu \ge 0$ we have $a \le 0$ and therefore (for $\lambda\delta < \lambda-\frac{a}{2}$; since $\lambda > 0$)
\[ \begin{split}
\bbE^\nu[ \exp(\lambda C_t) (1+C_t)^{-p} \,1_{T_0 > t} 1_{C_t \le \delta t}] 
&\le \exp(\lambda \delta t) \\
&\lesssim (\sin\theta_0)^a \exp\left((\lambda-\frac{a}{2})t\right) t^{-p} .
\end{split} \]
In case $\nu < 0$, applying \cref{prop:bessel_clock_moments_nul0}, we get (for $\lambda\delta+\nu < \lambda-\frac{a}{2}$)
\[ \begin{split}
\bbE^\nu[ \exp(\lambda C_t) (1+C_t)^{-p} \,1_{T_0 > t} 1_{C_t \le \delta t}] 
&\le \exp(\lambda \delta t)\,\bbE^\nu[1_{T_0 > t}] \\
&\le \exp(\lambda \delta t) (\sin\theta_0)^{-2\nu} e^{\nu t} \, \bbE^{-\nu}[(\sin\theta_t)^{2\nu}] \\
&\lesssim (\sin\theta_0)^a \exp\left((\lambda-\frac{a}{2})t\right) t^{-p}  
\end{split} \]
where we have used $-2\nu \ge a$ and $\nu < \lambda-\frac{a}{2}$.
\end{proof}

In the case of critical index $\nu=0$, the Bessel process almost hits the boundary, and the Bessel clock has much heavier tails.

\begin{proposition}\label{pr:bes_clock_critical}
If $\nu = 0$, then there exists $c>0$ such that
\[ \bbP^{\nu=0}(C_t > rt^2) \asymp_c (1+\abs{\log\sin\theta_0}/t)\,r^{-1/2} \]
for $r \ge (1+\abs{\log\sin\theta_0}/t)^2$ and $t\ge 1$.
\end{proposition}

We prove this by comparing it to the usual Bessel process.

\begin{lemma}\label{le:usual_bessel_clock}
Let $\rho_t$ be a (usual) Bessel process of index $0$, i.e.
\[ d\rho_t = dB_t+\frac{1}{2}\rho_t^{-1}\,dt , \]
and $\sigma_R = \inf\{t \mid \rho_t = R\rho_0\}$, then
\[ \int_0^{\sigma_R} \rho_t^{-2}\,dt \]
has the same law as $T_{\log R} = \inf\{t \mid B_t=\log R\}$ (with $B$ a standard Brownian motion with $B_0=0$), i.e.\@ it has the density
\[ f(r) = \frac{\abs{\log R}}{\sqrt{2\pi r^3}} \exp\left( -\frac{(\log R)^2}{2r} \right) . \]
\end{lemma}

\begin{proof}
Write $\tilde C_t = \int_0^t \rho_s^{-2}\,ds$. By the skew-product representation of $2$d Brownian motion (cf.\@ \cite[Theorem V.2.11]{ry-stochastic-book}), we may write $\rho_t = \rho_0\exp(\beta_{\tilde C_t})$ with a Brownian motion $\beta$. Since $\tilde C_t$ is continuously increasing, we have $\tilde C_{\sigma_R} = T_{\log R}$. The claim follows.
\end{proof}

\begin{proof}[Proof of \cref{pr:bes_clock_critical}]
\textbf{Step 1.} We first prove the case $t=1$.
By symmetry, we can assume $\theta_0 \le \pi/2$. Consider stopping times
\begin{align*}
\tau_1 &= \inf\{ t \ge 0 \mid \theta_t \ge \pi-1 \},\\
\tau_2 &= \inf\{ t \ge \tau_1 \mid \theta_t \le 1 \},\\
\tau_3 &= \inf\{ t \ge \tau_2 \mid \theta_t \ge \pi-1 \},\\
&\cdots
\end{align*}
and so on. There is at most a geometrically distributed number of such excursions on the time interval $[0,1]$. For each such excursion, we can compute the density of $\theta$ with respect to the usual Bessel process.

Let $d\rho_t = dB_t+\frac{1}{2}\rho_t^{-1}\,dt$ be a usual Bessel process, and $m(x) = \frac{1}{2}(\cot(x)-x^{-1})$. Let $\Phi(x) = \exp\left( \int_0^x m(y)\,dy \right)$ be the solution of $\Phi'(x)=m(x)\Phi(x)$, $\Phi''(x)=(m(x)^2+m'(x))\Phi(x)$. We have $m(x) = O(x)$ and $m'(x) = O(1)$ as $x \searrow 0$. Let
\[ M_t = \frac{\Phi(\rho_t)}{\Phi(\rho_0)}\exp\left( -\frac{1}{2} \int_0^t (m(\rho_s)\rho_s^{-1}+m(\rho_s)^2+m'(\rho_s))\,ds \right) . \]
Then under the probability measure $M_\tau\cdot\bbP$, the process $\rho$ has the law of a radial Bessel process. Moreover, as long as $\rho_s \le \pi-1$ (say), the integrands in $M_t$ are bounded.

Using this martingale, we can now compare to \cref{le:usual_bessel_clock} via a change of measure. Supposing $\rho_0 = \theta_0 \le 1$, we have
\[ \begin{split}
\bbP\left( \int_0^{\tau_1 \wedge 1} (\sin\theta_s)^{-2} \,ds > r \right)
&\le \bbP\left( \int_0^{\tau_1} \theta_s^{-2} \,ds > r-2 \right) \\
&= \bbE\left[ M_{\tau_1}\, 1_{\int_0^{\tau_1} \rho_s^{-2} \,ds > r-2} \right] \\
&\asymp \bbP\left( \int_0^{\tau_1} \rho_s^{-2} \,ds > r-2 \right) \\
&\asymp (\log\frac{2}{\rho_0}) r^{-1/2} .
\end{split} \]
If the process makes $n$ such excursions within time $1$, then $\int_0^1 (\sin\theta_s)^{-2} \,ds > r$ implies $\int_{\tau_{i-1} \wedge 1}^{\tau_i \wedge 1} (\sin\theta_s)^{-2} \,ds > r/n$ for some $i$. Therefore, by the strong Markov property,
\[ \begin{split}
\bbP\left( \int_0^1 (\sin\theta_s)^{-2} \,ds > r \right) 
&\le \sum_{n \in \bbN} p^n n \,\bbP\left( \int_0^{\tau_1 \wedge 1} (\sin\theta_s)^{-2} \,ds > r/n \right) \\
&\lesssim (\log\frac{2}{\theta_0}) r^{-1/2} \end{split} \]
which proves the upper bound.

For the lower bound it suffices to assume $\theta_0 = \rho_0 \le 1/2$. It remains to show that
\[ \bbP\left( \int_0^{{\tau_1} \wedge 1} \rho_s^{-2} \,ds > r \right) \asymp \bbP\left( \int_0^{\tau_1} \rho_s^{-2} \,ds > r \right) . \]
Recall that
\[ \bbP_{\rho_0}\left( \int_0^{\tau_1} \rho_s^{-2} \,ds > 2r \right) = \int_r^\infty \frac{2^{-1/2}\log\frac{\pi-1}{\rho_0}}{\sqrt{2\pi v^3}} \exp\left( -\frac{(2^{-1/2}\log\frac{\pi-1}{\rho_0})^2}{2v} \right) \,dv . \]
Observe that
\[ \int_0^{\tau_1} \rho_s^{-2} \,ds = \int_0^{{\tau_1}\wedge 1} \rho_s^{-2} \,ds + \int_{{\tau_1}\wedge 1}^{\tau_1} \rho_s^{-2} \,ds \]
and that
\[ \bbP_{\rho_0}\left( \int_{{\tau_1}\wedge 1}^{\tau_1} \rho_s^{-2} \,ds > r \mmiddle| \mathcal F_1 \right) = 1_{{\tau_1}>1} \bbP_{\rho_1}\left( \int_0^{\tau_1} \rho_s^{-2} \,ds > r \right) . \]
Therefore, recalling that the transition probability of the Bessel process is $p_1(0,y) = y\exp(-y^2/2)$ (cf.\@ \cite[p.\@ 446]{ry-stochastic-book}), we have
\[ \begin{split}
\bbP_{\rho_0}\left( \int_{{\tau_1}\wedge 1}^{\tau_1} \rho_s^{-2} \,ds > r \right) 
&\le \bbP_{0}\left( \int_{{\tau_1}\wedge 1}^{\tau_1} \rho_s^{-2} \,ds > r \right) \\
&\le \int_0^{\pi-1} p_1(0,y)\, \bbP_y\left( \int_0^{\tau_1} \rho_s^{-2} \,ds > r \right)\,dy \\
&= \int_r^\infty \int_0^{\pi-1} y\exp(-y^2/2) \\
&\hphantom{= \int_r^\infty \int_0^{\pi-1}} \frac{\log\frac{\pi-1}{y}}{\sqrt{2\pi v^3}} \exp\left( -\frac{(\log\frac{\pi-1}{y})^2}{2v} \right) \,dy\,dv \\
&< \frac{9}{10}\,\bbP_{1/2}\left( \int_0^{\tau_1} \rho_s^{-2} \,ds > 2r \right) 
\end{split} \]
where the last inequality can be seen by a direct computation.
It follows that with probability at least $\frac{1}{10}\,\bbP_{\rho_0}\left( \int_0^{\tau_1} \rho_s^{-2} \,ds > 2r \right)$ we have
\[ \int_0^{{\tau_1}\wedge 1} \rho_s^{-2} \,ds = \int_0^{\tau_1} \rho_s^{-2} \,ds - \int_{{\tau_1}\wedge 1}^\tau \rho_s^{-2} \,ds > 2r-r = r . \]
This proves the lower bound.

\textbf{Step 2.} We deduce the result of \cref{pr:bes_clock_critical} for $t \ge 1$ from the statement for $t=1$. This is essentially the well-known fact that sums of i.i.d.\@ centred random variables with infinite second moment behave like their maximum.

It suffices to assume $t \in \bbN$. The lower bound follows immediately from the result for $t=1$ since (applying the conclusion to $C_n-C_{n-1}$ and estimating $\abs{\log\sin\theta_{n-1}} \ge 0$)
\[ \begin{split}
\bbP( C_t \le rt^2 )
&\le \bbP( C_n-C_{n-1} \le rt^2 \text{ for all } n=1,...,t ) \\
&\le (1-c(1+\abs{\log\sin\theta_{0}})\,r^{-1/2}t^{-1})(1-cr^{-1/2}t^{-1})^{t-1} \\
&\le \exp\left(-c(1+\abs{\log\sin\theta_0}/t)\,r^{-1/2} \right) \\
&\le 1-\tilde c(1+\abs{\log\sin\theta_0}/t)\,r^{-1/2}
\end{split} \]
given that $r \ge (1+\abs{\log\sin\theta_0}/t)^2$.

To prove the upper bound, we first bound their maximal increment. Applying the result for $t=1$ and \cref{prop:bessel_density}, we have
\begin{align}
\bbP( C_n-C_{n-1} > rt^2 \text{ for some } n=1,...,t ) 
&\lesssim \sum_{n=1,...,t} \bbE(1+\abs{\log\sin\theta_{n-1}})\,r^{-1/2}t^{-1} \nonumber\\
&\lesssim (1+\abs{\log\sin\theta_0}/t)\,r^{-1/2} .\label{eq:max_upper}
\end{align}
Next, we show that on the event where the maximal increment is bounded, the sum is not likely exceed the bound either.
For this, we consider the random variables
\[ I_n \defeq (C_n-C_{n-1})1_{C_n-C_{n-1} \le rt^2} . \]
Applying again the result for $t=1$ and \cref{prop:bessel_density}, we get
\[ \begin{split}
\bbE I_n &\lesssim \bbE \int_0^{rt^2} \bbP( C_n-C_{n-1} > x \mid \cF_{n-1} ) \,dx \\
&\lesssim \bbE (1+\abs{\log\sin\theta_{n-1}})\, (rt^2)^{1/2} \\
&\lesssim \begin{cases}
r^{1/2}t & \text{for }n>1,\\
(1+\abs{\log\sin\theta_{0}})\,r^{1/2}t & \text{for }n=1.
\end{cases}
\end{split} \]
Hence,
\[ \bbP\left( \sum_{n=1,...,t} I_n > rt^2 \right) \le r^{-1}t^{-2}\sum_{n=1,...,t} \bbE I_n \lesssim (1+\abs{\log\sin\theta_0}/t)\,r^{-1/2} . \]
Combining this with \eqref{eq:max_upper}, this shows the upper bound.
\end{proof}

\section{Regularity of the trace}
\label{sec:regularity_proofs}

\subsection{Warmup: Existence of the trace}
\label{se:pf_warmup}

In order to illustrate the general idea of our proof, let us give a simple proof showing \slek{}, $\kappa \neq 8$, generates a continuous trace. The (more technical) proofs that come later are based on the idea that we describe in the following. We remark that the content of this subsection is mainly for illustration, and not required for the rest of the paper (although it greatly helps understanding what comes after).

By \cite[Corollary 3.12]{vl-sle-hoelder}, in order to have a continuous trace, it suffices to show $\abs{\hat f_t'(iv)} \le v^{-\beta}$ for some $\beta < 1$ (uniformly in $t \in [0,T]$ and small $v$). Due to Koebe's distortion theorem, it suffices to show this for $v=e^{-m}$, $m \in \bbN$.

We restrict to the set $\{ \norm{\xi}_\infty \le M \}$. Fix $m \in \bbN$. Suppose that $\abs{\hat f_t'(ie^{-m})} \ge e^{\beta m}$ for some $t \in [0,T]$. By \cref{thm:fw_grid}, there exists $z \in H(e^{-(1-\beta)m},M,T)$ such that $\abs{Z_t(z)-ie^{-m}} \le e^{-m}/2$ and $\Upsilon_t(z) \gtrsim e^{-(1-\beta)m}$. Recalling the parametrisation by conformal radius, this means $t \le \sigma(\delta m)$ where $\delta=\frac{1-\beta}{4}$. Therefore
\begin{multline}\label{eq:simple_pf_sum}
\bbP( \abs{\hat f_t'(ie^{-m})} \ge e^{\beta m} \text{ for some } t \in [0,T] ) \\
\le \sum_{z \in H(e^{-(1-\beta)m},M,T)} \bbP\left( Y_{\sigma(s)}(z) \asymp_c e^{-m} \text{ and } \frac{\abs{X_{\sigma(s)}}}{Y_{\sigma(s)}} \le 1 \text{ for some } s \le \delta m \right) . 
\end{multline}
If the sum of the probabilities decays exponentially in $m$, then by Borel-Cantelli we are done.

For $z=x+iy$, recall from \cref{se:prelim_general} that
\[ Y_{\sigma(s)}(z) = y \exp\left( -2 \int_{s_0}^s (\sin\hat\theta_s)^{-2} \, ds \right) \]
and from \cref{se:prelim_bm} that $\hat\theta$ is a radial Bessel process of index $\nu = \frac{1}{2}-\frac{4}{\kappa}$ started at $\hat\theta_{s_0} = \cot^{-1}(x/y)$ (where $s_0 = -\frac{1}{4}\log y$) and run at speed $\kappa s$. In particular, we can write
\[ Y_{\sigma(s)}(z) = y \exp\left( -\frac{2}{\kappa} C_{\kappa(s-s_0)} \right) \]
where $C$ denotes the radial Bessel clock defined in \cref{se:bessel}.

For $\kappa \neq 8$, we have $\nu \neq 0$. Therefore the probability on the right-hand side of \eqref{eq:simple_pf_sum} can be estimated by \cref{prop:bessel_clock_mixed_moments}. Let
\[ \tau \defeq \inf\left\{ s \in [s_0,\delta m] \mid Y_{\sigma(s)}(z) \asymp_c e^{-m} \text{ and } \hat\theta_s \in [\cot^{-1}(\pm 1)] \right\} \wedge (\delta m+1) . \]
Fix some $\lambda \in {]0,\frac{\nu^2}{2}[}$. We have
\[ \begin{split}
\bbP( \tau \le \delta m ) 
&\asymp y^{-\lambda}e^{-\lambda m} \bbE\left[ \exp\left( \frac{2\lambda}{\kappa} C_{\kappa(\tau-s_0)} \right) 1_{\tau \le \delta m} \right] \\
&\lesssim y^{-\lambda}e^{-\lambda m} (\sin\cot^{-1}(x/y))^a \exp\left(\left(\frac{2\lambda}{\kappa}-\frac{a}{2}\right)^+ \kappa(\delta m-s_0)\right) \\
&= e^{-\lambda m} \exp\left((2\lambda-\frac{a\kappa}{2})^+ \delta m \right) y^{-\lambda+(\lambda/2-a\kappa/8)^+}(1+\abs{x}/y)^{-a}
\end{split} \]
where $a=-\nu+\sqrt{\nu^2-\frac{4\lambda}{\kappa}}$ and $\eta^+ = \eta \vee 0$ denotes the positive part.

Picking $\beta$ close to $1$ (i.e.\@ $\delta$ close to $0$), we see that after summing in $z \in H(e^{-(1-\beta)m})$ that (according to \cref{le:fw_sum_grid}) the sum \eqref{eq:simple_pf_sum} is bounded by $e^{-\lambda m + \varepsilon m}$ where we can make $\varepsilon > 0$ as small as we want. This is summable in $m$, which is exactly what we wanted to show.

\subsection{Setup of our proofs}
\label{se:pf_setup}

We turn to the proofs of the main results of the paper. They follow the same idea as the previous subsection, but require much more care. We begin by discussing the technical setup. Recall the notations $\Upsilon_t$, $\sigma$, and $H(h,M,T)$ introduced in \cref{se:prelim_general}.

Let $v \in {]0,1]}$, $t \in [0,T]$, and find $\bar s = \bar s(t,v) \in \bbN$ such that $\Upsilon_t(\hat f_t(iv)) \in [e^{-4\bar s},e^{-4(\bar s-1)}]$. By \cref{le:distortion,thm:fw_grid}, there exists $z \in H(e^{-4\bar s},\norm{\xi},T)$ with $\Upsilon_t(z) \in [\frac{27}{160}e^{-4\bar s}, \frac{125}{32}e^{-4(\bar s-1)}]$ and $\abs{Z_t(z)-iv} \le v/2$. For $v=2^{-m}$, call this point $z(t,m)$.

By construction, $t = \sigma(s,z(t,m))$ for some $s = -\frac{1}{4}\log\Upsilon_t(z) \in [\bar s-2, \bar s+1]$.

For $z \in \bbH$ and $\bar s \in \bbN$, we consider the set
\[
P(z,\bar s) \defeq \left\{ (m,t) \in \bbN \times [0,T] \mmiddle| \begin{array}{cc}
\abs{Z_{\sigma(s,z)}(z) - i2^{-m}} \le 2^{-m-1} \text{ and } \\
\sigma(s,z) = t \text{ for some } s \in [\bar s-2, \bar s+1] 
\end{array} \right\} .
\]
Let $P'(z,\bar s)$ be any subset of $P(z,\bar s)$ such that any two $(m,t_1),(m,t_2) \in P'(z,\bar s)$ satisfy $\abs{t_1-t_2} \ge 2^{-2m}$. Let $N(z,\bar s)$ be the largest possible cardinality of such $P'(z,\bar s)$.

Let us remark here that if the process $\hat\theta$ dies before time $\bar s-2$, then $N(z,\bar s) = 0$.

To count $N(z,\bar s)$, we define a sequence of stopping times $S_n = S_n(z,\bar s)$, $T_n = T_n(z,\bar s)$ as follows. Fix some $b > 1$ (the exact value does not matter). Let
\[ S_0 \defeq \inf\left\{ s \in [\bar s-2,\bar s+1] \mmiddle| \frac{X_{\sigma(s)}^2}{Y_{\sigma(s)}^2} \le 1 \right\} , \]
and inductively
\begin{align*}
T_n &\defeq \inf\left\{ s \in [S_n,\bar s+1] \mmiddle| \frac{X_{\sigma(s)}^2}{Y_{\sigma(s)}^2} \ge b \right\} \wedge (\bar s+1) ,\\
S_{n+1} &\defeq \inf\left\{ s \in {]T_n,\bar s+1]} \mmiddle| \frac{X_{\sigma(s)}^2}{Y_{\sigma(s)}^2} \le 1 \right\} .
\end{align*}
Let
\[
P_n(z,\bar s) \defeq \left\{ (m,t) \in \bbN \times [0,T] \mmiddle| \begin{array}{cc} 
\abs{Z_{\sigma(s,z)}(z) - i2^{-m}} \le 2^{-m-1} \text{ and } \\
\sigma(s,z) = t \text{ for some } s \in [S_n,T_n] \end{array} \right\} ,
\]
and note that $P(z,\bar s) = \bigcup_n P_n(z,\bar s)$. Similarly to above, let $P'_n(z,\bar s)$ be any subset of $P_n(z,\bar s)$ such that any two $(m,t_1),(m,t_2) \in P'_n(z,\bar s)$ satisfy $\abs{t_1-t_2} \ge 2^{-2m}$.

 Let $N_n(z,\bar s)$ be the largest possible cardinality of such $P'_n(z,\bar s)$.

Moreover, note that $\frac{X_{\sigma(s)}^2}{Y_{\sigma(s)}^2} \le b$ on any interval $[S_n,T_n]$. Letting
\[ p = \bbP\left( \left. \frac{X_{\sigma(s)}^2}{Y_{\sigma(s)}^2} > b \text{ for some } s \in [0,3]\ \right|\ \frac{X_{\sigma(0)}^2}{Y_{\sigma(0)}^2} = 1 \right) \in {]0,1[} , \]
we see that $\bbP( S_n < \infty ) \le p^n$.

\begin{lemma}\label{le:max_num_times}
There exists $N \in \bbN$ such that $N_n(z,\bar s) \le N$ for any $z,\bar s,n$.
\end{lemma}

\begin{proof}
Write $\tilde b = 1+b$. If $1+\frac{X_{\sigma(s)}^2}{Y_{\sigma(s)}^2} \le \tilde b$, then by \eqref{eq:Ysigma_dynamics} we have
$dY_{\sigma(s)}^2 \ge -4\tilde b Y_{\sigma(s)}^2 \,ds$
and hence (by Grönwall's inequality)
\[ Y_{\sigma(s)} \ge e^{-2\tilde b(s-\bar s)}Y_{\sigma(\bar s)} . \]
Moreover, by \eqref{eq:sigma_dynamics} we have
$d\sigma(s) \le \tilde b^2 Y_{\sigma(s)}^2 \,ds$
and hence 
\[ \sigma(s) \le \sigma(\bar s)+(s-\bar s)\tilde b^2 Y_{\sigma(\bar s)}^2 . \]

Suppose now that $(m_0,t_0) \in P_n(z,\bar s)$, i.e.\@ $\abs{Z_{\sigma(s,z)}(z) - i2^{-m_0}} \le 2^{-m_0-1}$ and $\sigma(s,z) = t_0$ for some $s \in [S_n,T_n]$. In particular, $Y_{\sigma(s)} \in [2^{-m_0-1}, 2^{-m_0+1}]$.

If we find another pair $(m,t) \in P_n(z,\bar s)$, then (by our previous observation) we must have $2^{-m+1} \ge e^{-6\tilde b}2^{-m_0-1}$ and $t \le t_0+3\tilde b^2 2^{-2m_0+2}$. But there is a fixed maximum number $N$ of such pairs $(m,t)$ where the $t$ also have distance at least $2^{-2m}$ from each other (and that number $N$ does not depend on $m_0$ or $n$).
\end{proof}

We are going to need one more addition to this, the reason of which will become apparent at the end of the proof of \cref{thm:gvar,thm:hoelder}. Recall $\xi(t) = \sqrt{\kappa}B_t$.

\begin{lemma}\label{lem:brownian_increment}
Let $z\in \bbH$, $\bar s \in \bbN$. For every $n$ there exists a random variable $M_n(z,\bar s)$ that is independent of $\mathcal F_{\sigma(S_n)}$ and such that for every $(m,t) \in P_n(z,\bar s)$ and $u > t$ with $\abs{u-t} \in [2^{-2(m+1)},2^{-2(m-1)}]$ we have
\[ \frac{\abs{\xi(u)-\xi(t)}}{\abs{u-t}^{1/2}} \le M_n(z,\bar s) . \]
Moreover, each $M_n(z,\bar s)$ has the same law and has all exponential moments.
\end{lemma}

\begin{proof}
Let
\[ M_n(z,\bar s) \defeq \sup_{t,u} \frac{\abs{\xi(u)-\xi(t)}}{\abs{u-t}^{1/2}} \]
where the supremum runs over $t<u$ with $t \in [\sigma(S_n),\sigma(S_n)+3\tilde b^2Y_{\sigma(S_n)}^2]$ and $\abs{u-t} \in [\frac{1}{16}e^{-12\tilde b}Y_{\sigma(S_n)}^2, 16Y_{\sigma(S_n)}^2]$.

By the strong Markov property and Brownian scaling, we see that each $M_n$ has the same law, and finite exponential moments.

For every $(m,t) \in P_n(z,\bar s)$, by \cref{le:max_num_times}, we have $t \in [\sigma(S_n), \sigma(T_n)] \subseteq [\sigma(S_n), \sigma(S_n)+3\tilde b^2Y_{\sigma(S_n)}^2]$ and $2^{-m} \in [\frac{1}{2}Y_{t}, 2Y_{t}] \subseteq [\frac{1}{2}e^{-6\tilde b}Y_{\sigma(S_n)}, 2Y_{\sigma(S_n)}]$. If $u > t$ and $\abs{u-t} \in [2^{-2(m+1)},2^{-2(m-1)}]$, then also $\abs{u-t} \in [\frac{1}{16}e^{-12\tilde b}Y_{\sigma(S_n)}^2, 16Y_{\sigma(S_n)}^2]$. In particular, the term
\[ \frac{\abs{\xi(u)-\xi(t)}}{\abs{u-t}^{1/2}} \]
appears in the supremum defining $M_n$.
\end{proof}

\subsection{Generalised variation}
\label{se:gvar_pf}

In this section, we are going to estimate the $\psi$-variation of the SLE trace.

We will frequently use the following estimate. Let $\psi$ be a convex function and $p_n \ge 0$ a summable sequence with $p = \sum p_n < \infty$. By Jensen's inequality we have
\[ \psi\left( \sum a_n \right) = \psi\left( \sum p\frac{a_n}{p_n} \frac{p_n}{p} \right) \le \sum \psi\left(p\frac{a_n}{p_n}\right) \frac{p_n}{p} . \]

In the following, we will assume that $\psi$ is convex and satisfies the condition ($\Delta_c$) (see \cref{se:prelim_psi_var}).

Now let $0 = t_0 < t_1 < ... < t_r = T$ be a partition of $[0,T]$. Recall the notation from \cref{se:pf_setup}. For $z \in \bbH$ and $\bar s \in \bbN$, the following sets of pairs
\begin{align*}
&\{ (m,t_j) \mid (m,t_j) \in P(z,\bar s) \text{ and } \abs{t_j-t_{j-1}} \ge 2^{-2m} \} , \\
&\{ (m,t_j) \mid (m,t_j) \in P(z,\bar s) \text{ and } \abs{t_{j+1}-t_j} \ge 2^{-2m} \}
\end{align*}
each form a set $P'(z,\bar s)$ as described in \cref{se:pf_setup}.

\begin{lemma}\label{lem:gvar_bound}
Let $M,T > 0$ and $\varepsilon > 0$. There exists $C > 0$ depending on $\psi,T,\varepsilon$ such that if $\norm{\xi}_{\infty;[0,T]} \le M$, then
\begin{multline*}
\sum_j \psi\left( \abs{\gamma(t_j)-\hat f_{t_j}(i\abs{t_j-t_{j-1}}^{1/2})} \right) + \psi\left( \abs{\gamma(t_j)-\hat f_{t_j}(i\abs{t_{j+1}-t_j}^{1/2})} \right) \\
\le C \sum_{s \in \bbN} \sum_{z \in H(e^{-4s},M,T)} \sum_{n \in \bbN_0} \hfill\\
(\logp Y_{\sigma(S_n,z)}(z)^{-1})^{-1-\varepsilon} \psi\left( e^{-4s} (\logp Y_{\sigma(S_n,z)}(z)^{-1})^{1+\varepsilon} \right) 1_{S_n(z,s) < \infty}
\end{multline*}
for any partition of $[0,T]$.
\end{lemma}

\begin{remark}
This is almost an upper bound on the $\psi$-variation of $\gamma$. Note that the right-hand side does not depend on the choice of the partition.
\end{remark}

\begin{proof}
Pick $m_j \in \bbN$ with $2^{-m_j} \asymp_2 \abs{t_j-t_{j-1}}^{1/2}$. By \eqref{eq:f_int}, we have
\[ \abs{\gamma(t)-\hat f_t(i2^{-m_j})} \lesssim \sum_{m \ge m_j} \Upsilon_t(\hat f_t(i2^{-m})) . \]
Applying Jensen's inequality as above (and the assumption ($\Delta_c$) for $\psi$) yields
\[ \psi\left( \abs{\gamma(t)-\hat f_t(i2^{-m_j})} \right) \lesssim \sum_{m \ge m_j} m^{-1-\varepsilon} \psi\left( \Upsilon_t(\hat f_t(i2^{-m})) m^{1+\varepsilon} \right) . \]

Applying this to $2^{-m_j} \asymp \abs{t_j-t_{j-1}}^{1/2}$ and summing over $t_j$, we get
\begin{equation}\label{eq:gvar_crad}
\sum_j \psi\left( \abs{\gamma(t_j)-\hat f_{t_j}(i\abs{t_j-t_{j-1}}^{1/2})} \right) 
\lesssim \sum_{m \in \bbN} \sum_{j \in I_m}  m^{-1-\varepsilon} \psi\left( \Upsilon_{t_j}(\hat f_{t_j}(i2^{-m})) m^{1+\varepsilon} \right)
\end{equation}
where $I_m = \{i \mid \abs{t_j-t_{j-1}}  \ge 2^{-2m}\}$.

The same applies with $\abs{t_j-t_{j-1}}$ replaced by $\abs{t_{j+1}-t_j}$, so we can just focus on the former.

We rearrange the sum \eqref{eq:gvar_crad} by collecting (for each $s \in \bbN$) those terms where \linebreak$\Upsilon_{t_j}(\hat f_{t_j}(i2^{-m})) \in [e^{-4s},e^{-4(s-1)}]$. As we observed in \cref{se:pf_setup}, we can find $z = z(t_j,m) \in H(e^{-4s},\norm{\xi},T)$ such that $\Upsilon_{t_j}(z) \asymp \Upsilon_{t_j}(\hat f_{t_j}(i2^{-m}))$ and $\abs{Z_{t_j}(z) - i2^{-m}} \le 2^{-m-1}$.

In particular, we have $Y_{t_j}(z) \asymp 2^{-m}$ and $\Upsilon_{t_j}(z) \asymp \Upsilon_{t_j}(\hat f_{t_j}(i2^{-m})) \asymp e^{-4s}$, and $(m,t_j) \in P_n(z,s)$ for some $n$. Moreover, by \cref{le:max_num_times}, for each choice of $s$, $z$, and $n$, we can have at most $N$ pairs of $(m,j)$ with $(m,t_j) \in P_n(z,s)$. Finally, we have shown there also that $Y_{t_j} \asymp Y_{\sigma(S_n)}$. Putting everything together, we get
\[ \begin{split}
&\sum_j \psi\left( \abs{\gamma(t_j)-\hat f_{t_j}(i\abs{t_j-t_{j-1}}^{1/2})} \right) \\
&\quad \lesssim \sum_{s \in \bbN} \sum_{m \in \bbN} \sum_{j \in I_m} 1_{\Upsilon_{t_j}(\hat f_{t_j}(i2^{-m})) \in [e^{-4s},e^{-4(s-1)}]} \\
&\quad \hphantom{\lesssim \sum_{s \in \bbN} \sum_{m \in \bbN} \sum_{j \in I_m}} 
(\logp Y_{t_j}(z(t_j,m))^{-1})^{-1-\varepsilon} \psi\left( e^{-4s} (\logp Y_{t_j}(z(t_j,m))^{-1})^{1+\varepsilon} \right) \\
&\quad \lesssim \sum_{s \in \bbN} \sum_{z \in H(e^{-4s},M,T)} \sum_{n \in \bbN_0} \\
&\quad \hphantom{\lesssim \sum_{s \in \bbN}}
(\logp Y_{\sigma(S_n,z)}(z)^{-1})^{-1-\varepsilon} \psi\left( e^{-4s} (\logp Y_{\sigma(S_n,z)}(z)^{-1})^{1+\varepsilon} \right) 1_{S_n(z,s) < \infty} .
\end{split} \]
\end{proof}

In the following, we consider the function $\psi_{p,q}$ defined in \eqref{eq:psi_pq}. By \eqref{eq:psi_pq_est}, we can estimate
\[ \psi\left( e^{-4s} (\logp Y_{\sigma(s)}(z)^{-1})^{1+\varepsilon} \right) \lesssim \big(...\big)^p s^{-q} \left(\logp\logp Y_{\sigma(s)}(z)^{-1}\right)^{q} . \]
Recall that by definition $S_n(z,s) \in [s-2,s+1]$ whenever it is finite.
Hence, using also $\logp\logp x \lesssim (\logp x)^{\varepsilon}$, we are reduced to estimate
\begin{equation}\label{eq:gvar_bound_pq} 
\sum_{s \in \bbN} \sum_{z \in H(e^{-4s},M,T)} \sum_{n \in \bbN_0} e^{-4ps} s^{-q} \left(\logp Y_{\sigma(S_n)}^{-1}\right)^{p-1+\varepsilon} 1_{S_n < \infty} .
\end{equation}
Recall that $\bbP(S_n < \infty) \le p^n$ for some $p < 1$. Therefore, when taking expectations, we can (using Hölder's inequality) safely ignore the sum in $n$ at the cost of a multiplicative factor.

\begin{remark}
In the expression \eqref{eq:gvar_bound_pq} we see again the phase transition of the $p$-variation exponent $d = (1+\frac{\kappa}{8}) \wedge 2$. Recall that $\hat\theta$ is a radial Bessel process of index $\nu = \frac{1}{2}-\frac{4}{\kappa}$ which can hit the boundary in case $\nu < 0 \iff \kappa < 8$. Consequently, the process has finite lifetime, and the probability of survival decays like $e^{\nu t}$. This allows the summand to be much smaller than $e^{-4ps}$ and therefore allows for a choice of $p<2$. In case $\kappa \ge 8$, the summand will not be smaller than $e^{-4ps}$, and since for each $s$ we have $e^{8s}$ summands, we need $p \ge 2$ to make the sum converge.
\end{remark}

Recall from \cref{se:prelim_general,se:prelim_bm} (and explained again in \cref{se:pf_warmup}) that we can write
\[ Y_{\sigma(s)}(z) = y \exp\left( -\frac{2}{\kappa} C_{\kappa(s-s_0)} \right) \]
where $C$ is the radial Bessel clock defined in \cref{se:bessel}, for a radial Bessel process of index $\nu = \frac{1}{2}-\frac{4}{\kappa}$ started at $\hat\theta_{s_0} = \cot^{-1}(x/y)$ (where $s_0 = -\frac{1}{4}\log y$). In other words, we have
\[ \log Y_{\sigma(s)}(z)^{-1} = \log\frac{1}{y} + \frac{2}{\kappa} C_{\kappa(s-s_0)} . \]
We can now apply the results from \cref{se:bessel}.

First consider the case $\kappa > 8 \iff \nu > 0$. By \cref{prop:bessel_clock_moments_nug0}, we have
\[ \begin{split}
\bbE (\logp Y_{\sigma(s)}(z)^{-1})^\eta &\lesssim (\logp (1/y))^\eta + (s-s_0)^\eta + (-\log\sin\cot^{-1}(x/y))^\eta \\
&\lesssim s^\eta+\left(\logp\frac{1}{y}+\log(1+\frac{\abs{x}}{y})\right)^\eta .
\end{split} \]

Hence,
\[ \begin{split}
&\bbE[\text{ eq. \eqref{eq:gvar_bound_pq} }]\\
&\quad \lesssim \sum_{s \in \bbN} \sum_{z \in H(e^{-4s},M,T)} e^{-4ps} s^{-q} \left( s^{p-1+\varepsilon}+\left(\logp\frac{1}{y}+\log(1+\frac{\abs{x}}{y})\right)^{p-1+\varepsilon} \right) \\
&\quad \asymp \sum_{s \in \bbN} s^{p-1-q+\varepsilon} e^{(8-4p)s}
\end{split} \]
which converges for $p=2$, $q>2$.

In case $\kappa < 8 \iff \nu = \frac{1}{2}-\frac{4}{\kappa} < 0$, we apply \cref{prop:bessel_clock_moments_nul0}. Using also that (by the definition of $S_n$) $S_n \in [s-2,s+1]$ and $\sin\hat\theta_{S_n} \ge \sin\cot^{-1}(1)$ whenever $S_n < \infty$, we get
\[ \begin{split}
&\bbE[(\logp Y_{\sigma(S_n)}(z)^{-1})^\eta \,1_{T_0 > S_n}] \\
&\quad \lesssim (\sin\cot^{-1}(x/y))^{-2\nu} \,\bbE^{-\nu}[(\logp\frac{1}{y}+C_{\kappa(s+1-s_0)})^\eta e^{\nu {\kappa(s-s_0)}}] \\
&\quad \lesssim (\sin\cot^{-1}(x/y))^{-2\nu-\varepsilon}(\logp\frac{1}{y})^\eta (s-s_0)^\eta e^{\nu\kappa (s-s_0)} \\
&\quad \lesssim s^\eta \exp((\frac{\kappa}{2}-4)s) (1+\abs{x}/y)^{1-8/\kappa+\varepsilon} y^{\kappa/8-1-\varepsilon} .
\end{split} \]
Hence, applying also \cref{le:fw_sum_grid}, we get
\[ \begin{split}
\bbE[\text{ eq. \eqref{eq:gvar_bound_pq} }] 
& \lesssim \sum_{s \in \bbN} \sum_{z \in H(e^{-4s},M,T)} \\
& \hphantom{\lesssim \sum_{s \in \bbN} \sum}
e^{-4ps} s^{-q} s^{p-1+\varepsilon} \exp((\frac{\kappa}{2}-4)s) (1+\abs{x}/y)^{1-8/\kappa+\varepsilon} y^{\kappa/8-1-\varepsilon} \\
& \asymp \sum_{s \in \bbN} s^{p-1-q+\varepsilon} e^{(4+\kappa/2-4p)s}
\end{split} \]
which converges for $p=1+\frac{\kappa}{8}$, $q>p=1+\frac{\kappa}{8}$.

\begin{theorem}\label{thm:gvar}
Let $\kappa \in {]0,8[} \cup {]8,\infty[}$. Let $d=(1+\frac{\kappa}{8}) \wedge 2$, $q>d$, and $\psi_{d,q}$ as in \eqref{eq:psi_pq}. Then, restricted to the event $\{\norm{\xi}_{[0,T]} \le M\}$, we have
\[ \bbE\left[ V^1_{\psi_{d,q};[0,T]}(\gamma) \ 1_{\{\norm{\xi}_{[0,T]} \le M\}} \right] < \infty . \]
In particular, $\bbE\left[ [\gamma]_{\psi_{d,q}\text{-var};[0,T]}^{p} \right] < \infty$ for any $p < d$.
\end{theorem}

\begin{proof}
First note that restricting to the event $\{\norm{\xi}_{[0,T]} \le M\}$ is enough since all our previous estimates depend on $M$ polynomially (the dependence comes from \cref{le:fw_sum_grid}), whereas the probability of $\bbP( \norm{\xi}_{[0,T]} > M )$ decays exponentially in $M$. The second claim then follows from the fact that $[\gamma]_{\psi_{d,q}\text{-var};[0,T]} \lesssim V^1_{\psi_{d,q};[0,T]}(\gamma)^{1/d}$.

So we are almost reduced to what we have estimated above.

Let $0 = t_0 < t_1 < ... < t_r = T$ be any partition of $[0,T]$. Write $v_j = \abs{t_{j+1}-t_j}^{1/2}$. By our assumption $\Delta_c$ on $\psi$,
\begin{multline*}
\psi(\abs{\gamma(t_{j+1})-\gamma(t_j)}) \\
\lesssim \psi\left( \abs{\gamma(t_{j+1})-\hat f_{t_{j+1}}(iv_j)} \right) 
+ \psi\left( \abs{\hat f_{t_{j+1}}(iv_j)-\hat f_{t_j}(iv_j)} \right) 
+ \psi\left( \abs{\gamma(t_j)-\hat f_{t_j}(iv_j))} \right) .
\end{multline*}
The sums over the first and the third term appear already in \cref{lem:gvar_bound} which we have bounded by \eqref{eq:gvar_bound_pq}. Recall that the expression \eqref{eq:gvar_bound_pq} does not depend on the choice of the partition, and that $\bbE[\text{ eq. \eqref{eq:gvar_bound_pq} }] < \infty$.

So we are left to estimate the middle term. We show that it is bounded by \eqref{eq:gvar_bound_pq} as well.

By \eqref{eq:f_diff}, we have
\[
\abs{\hat f_{t_{j+1}}(iv_j)-\hat f_{t_j}(iv_j)} 
\le \Upsilon_{t_j}(\hat f_{t_j}(iv_j)) \left( 1 + \left(\frac{\abs{\xi(t_{j+1})-\xi(t_j)}^2}{\abs{t_{j+1}-t_j}}\right)^l \right) .
\]

Pick $m_j \in \bbN$ with $2^{-m_j} \asymp_2 v_j$.

As in the proof of \cref{lem:gvar_bound}, we can collect the indices where $\Upsilon_{t_j}(\hat f_{t_j}(i2^{-m_j})) \in [e^{-4s},e^{-4(s-1)}]$, and as before, we can replace $\hat f_{t_j}(i2^{-m_j})$ by $z(t_j,m_j) \in H(e^{-4s},\norm{\xi},T)$, and we have $Y_{t_j}(z(t_j,m_j)) \asymp 2^{-m_j}$ and $\Upsilon_{t_j}(z(t_j,m_j)) \asymp \Upsilon_{t_j}(\hat f_{t_j}(i2^{-m_j})) \asymp e^{-4s}$, and $(m_j,t_j) \in P(z,s)$.

Finally, since $\abs{t_{j+1}-t_j} = v_j^2 \asymp_4 2^{-2m_j}$, by \cref{lem:brownian_increment}
\[ \left(\frac{\abs{\xi(t_{j+1})-\xi(t_j)}^2}{\abs{t_{j+1}-t_j}}\right)^l \le M_n(z,s)^{2l} . \]
Hence,
\[ \begin{split}
&\sum_j \psi(\abs{\hat f_{t_{j+1}}(iv_j)-\hat f_{t_j}(iv_j)}) \\
&\quad \le \sum_j \psi\left( \Upsilon_{t_j}(\hat f_{t_j}(iv_j)) \left( 1 + \left(\frac{\abs{\xi(t_{j+1})-\xi(t_j)}^2}{\abs{t_{j+1}-t_j}}\right)^l \right) \right) \\
&\quad \lesssim \sum_{s \in \bbN} \sum_{z \in H(e^{-4s},M,T)} \sum_{n \in \bbN_0} \psi(e^{-4s} (1+M_n(z,s)^{2l})) \,1_{S_n(z,s) < \infty} \\
&\quad \lesssim \sum_{s \in \bbN} \sum_{z \in H(e^{-4s},M,T)} \sum_{n \in \bbN_0} e^{-4ps} s^{-q} (1+M_n(z,s)^{2lp+\varepsilon}) \,1_{S_n(z,s) < \infty} \\
\end{split} \]
where we have applied \eqref{eq:psi_pq_est} in the last step. We see that this sum is also bounded by the expression \eqref{eq:gvar_bound_pq} except for a factor $(1+M_n(z,s)^{2lp})$. But when taking the expectation, that factor is irrelevant since it is independent of $\mathcal F_{\sigma(S_n)}$ and has finite exponential moments.
\end{proof}

\subsection{Hölder-type modulus}
\label{se:hoelder_pf}

We can estimate the Hölder-type modulus of \slek{} in a similar way as in the previous section.

Let $\psi,\varphi \colon {[0,\infty[} \to {[0,\infty[}$ be homeomorphisms that satisfy the condition ($\Delta_c$) resp.\@ ($\tilde\Delta_c$) with two different functions $\Delta_c$, $\tilde\Delta_c$ (see \cref{se:prelim_psi_var}; convexity is not required here).

\begin{lemma}\label{lem:hoelder_bound}
Let $M,T > 0$. There exists $C > 0$ depending on $\psi,\varphi,T$ such that if $\norm{\xi}_{\infty;[0,T]} \le M$, then
\begin{multline*}
\psi\left( \frac{ \abs{\gamma(t_1)-\hat f_{t_1}(i\abs{t_1-t_2}^{1/2})} }{\varphi(\abs{t_1-t_2})} \right) \\
\le C \sum_{s \in \bbN} \sum_{z \in H(e^{-4s},M,T)} \sum_{n \in \bbN_0} \psi\left( \frac{ e^{-4s} }{\varphi(Y_{\sigma(S_n,z)}(z)^{2})} \right) 1_{S_n(z,s) < \infty}
\end{multline*}
for any $t_1,t_2 \in [0,T]$.
\end{lemma}

\begin{remark}
The proof shows that for any partition of $[0,T]$, the sum
\[ \sum_j \psi\left( \frac{ \abs{\gamma(t_j)-\hat f_{t_j}(i\abs{t_j-t_{j-1}}^{1/2})} }{\varphi(\abs{t_j-t_{j-1}})} \right) + \psi\left( \frac{ \abs{\gamma(t_j)-\hat f_{t_j}(i\abs{t_j-t_{j+1}}^{1/2})} }{\varphi(\abs{t_j-t_{j+1}})} \right) \]
is bounded by the same expression on the right-hand side.
\end{remark}

\begin{proof}
The proof is similar to \cref{lem:gvar_bound}. The difference is that here we can use the condition ($\tilde\Delta_c$) for $\varphi$ (instead of Jensen's inequality) and get
\[ 
\frac{\sum_{m \ge m_j} \Upsilon_{t}(\hat f_{t}(i2^{-m}))}{\varphi(2^{-2m_j})}
\lesssim \max_{m\ge m_j} \frac{\Upsilon_{t}(\hat f_{t}(i2^{-m}))}{\varphi(2^{-2m})} .
\]
Indeed, the second condition of ($\tilde\Delta_c$) implies $\sum_{m \ge m_j} \frac{\varphi(2^{-2m})}{\varphi(2^{-2m_j})} \le C$ where $C$ does not depend on $m_j$.

In particular, we have the slightly improved estimate
\[ 
\psi\left( \frac{\sum_{m \ge m_j} \Upsilon_{t}(\hat f_{t}(i2^{-m}))}{\varphi(2^{-2m_j})} \right)
\lesssim \sum_{m\ge m_j} \psi\left( \frac{\Upsilon_{t}(\hat f_{t}(i2^{-m}))}{\varphi(2^{-2m})} \right) .
\]
Using this, the proof continues the same as in \cref{lem:gvar_bound}.
\end{proof}

In the following, let us consider $\psi(x) = x^p$ and $\varphi(t) = t^{\alpha} (\logp(\frac{1}{t}))^{\beta} \ell(\logp(\frac{1}{t}))^{1/p}$ where $\ell$ is as in \eqref{eq:ell_def}. The right-hand side of \cref{lem:hoelder_bound} then becomes
\begin{equation}\label{eq:hoelder_bound_pq} 
\sum_{s \in \bbN} \sum_{z \in H(e^{-4s},M,T)} \sum_{n \in \bbN_0} e^{-4ps} \, Y_{\sigma(S_n)}^{-2p\alpha} \left(\logp Y_{\sigma(S_n)}^{-1}\right)^{-p\beta} \ell\left(\logp Y_{\sigma(S_n)}^{-1}\right)^{-1} 1_{S_n < \infty} .
\end{equation}

Recall from \cref{se:prelim_general,se:prelim_bm} (and explained again in \cref{se:pf_warmup}) that we can write
\[ Y_{\sigma(s)}(z) = y \exp\left( -\frac{2}{\kappa} C_{\kappa(s-s_0)} \right) \]
where $C$ is the radial Bessel clock defined in \cref{se:bessel}, for a radial Bessel process of index $\nu = \frac{1}{2}-\frac{4}{\kappa}$ started at $\hat\theta_{s_0} = \cot^{-1}(x/y)$ (where $s_0 = -\frac{1}{4}\log y$).

We need to estimate
\begin{equation}\label{eq:Y_mixed_moments}
\bbE\left[ Y_{\sigma(S_n)}^{-2p\alpha} \left(\logp Y_{\sigma(S_n)}^{-1}\right)^{-p\beta} \ell\left(\logp Y_{\sigma(S_n)}^{-1}\right)^{-1} 1_{S_n < \infty} \right] .
\end{equation}

Similarly as in \cref{prop:bessel_clock_mixed_moments_neg}, we can split up into the events $\{ Y_{\sigma(S_n)} \le e^{-\delta S_n} \}$ and $\{ Y_{\sigma(S_n)} \ge e^{-\delta S_n} \}$ with suitable $\delta>0$.

Beginning with $\{ Y_{\sigma(S_n)} \le e^{-\delta S_n} \}$, we find $\logp Y_{\sigma(S_n)}^{-1} \gtrsim S_n$. Applying \cref{prop:bessel_clock_mixed_moments} and noting that (by the definition of $S_n$) $S_n \in [s-2,s+1]$ and $\sin\hat\theta_{S_n} \ge \sin\cot^{-1}(1)$ whenever $S_n < \infty$, we get
\[ \begin{split}
&\bbE\left[ Y_{\sigma(S_n)}^{-2p\alpha} \left(\logp Y_{\sigma(S_n)}^{-1}\right)^{-p\beta} \ell\left(\logp Y_{\sigma(S_n)}^{-1}\right)^{-1} 1_{S_n < \infty} \,1_{Y_{\sigma(S_n)} \le e^{-\delta S_n}} \right] \\
&\quad \lesssim s^{-p\beta}\ell(s)^{-1} y^{-2p\alpha} \bbE\left[ \exp\left( \frac{4p\alpha}{\kappa} C_{\kappa(S_n-s_0)} \right) 1_{S_n < \infty} \right] \\
&\quad \lesssim s^{-p\beta}\ell(s)^{-1} y^{-2p\alpha} (\sin\cot^{-1}(x/y))^{a} \exp\left(\left(\frac{4p\alpha}{\kappa}-\frac{a}{2}\right)\kappa(s-s_0) \right) \\
&\quad \lesssim s^{-p\beta}\ell(s)^{-1} \exp\left((4p\alpha-\frac{a\kappa}{2})s \right) y^{-2p\alpha+p\alpha-a\kappa/8} (1+\abs{x}/y)^{-a} 
\end{split} \]
where $a = -\nu+\sqrt{\nu^2-\frac{8}{\kappa}p\alpha}$.

On the event $\{ Y_{\sigma(S_n)} \ge e^{-\delta S_n} \}$ we have
\[ \begin{split}
&\bbE\left[ Y_{\sigma(S_n)}^{-2p\alpha} \left(\logp Y_{\sigma(S_n)}^{-1}\right)^{-p\beta} \ell\left(\logp Y_{\sigma(S_n)}^{-1}\right)^{-1} 1_{S_n < \infty} \,1_{Y_{\sigma(S_n)} \ge e^{-\delta S_n}} \right] \\
&\quad \lesssim e^{\delta 2p\alpha s}\, \bbE[1_{S_n < \infty}] \\
&\quad \lesssim s^{-p\beta}\ell(s)^{-1} \exp\left(\left(\frac{4p\alpha}{\kappa}-\frac{a}{2}\right)\kappa s \right) (\sin\cot^{-1}(x/y))^{a} \\
&\quad \lesssim s^{-p\beta}\ell(s)^{-1} \exp\left((4p\alpha-\frac{a\kappa}{2})s \right) (1+\abs{x}/y)^{-a} .
\end{split} \]
This is true whenever $\frac{8}{\kappa}p\alpha \in {]0,\nu^2[}$, or equivalently $a \in {]-\nu, (-2\nu) \vee 0[}$.

The sum in $n$ is harmless again since the estimate $\bbP^{\nu+a}(S_n < \infty) \le \tilde p^n$ is still true in the changed measure. Indeed, since $S_n$ is defined in terms of $\hat\theta_s = \cot^{-1}(X_{\sigma(s)}/Y_{\sigma(s)})$, it holds under any measure under which $\hat\theta$ is a (time-homogeneous) Markov process (with $\tilde p$ depending on the measure). This means that summing over $n$ just gives us an additional multiplicative factor.

With the estimate for \eqref{eq:Y_mixed_moments}, we are left to investigate the convergence of
\begin{multline*}
\bbE[\text{ eq. \eqref{eq:hoelder_bound_pq} }]
\lesssim \sum_{s \in \bbN} \sum_{z \in H(e^{-4s},M,T)} \\
\exp\left((-4p+4p\alpha-\frac{a\kappa}{2})s \right) s^{-p\beta}\ell(s)^{-1} y^{(-p\alpha-a\kappa/8) \wedge 0} (1+\abs{x}/y)^{-a} .
\end{multline*}
Since $\alpha,p,a$ are related by $\frac{4}{\kappa}p\alpha = -a(\frac{a}{2}+\nu) \iff p\alpha = -\frac{a^2\kappa}{8}-\frac{a\kappa}{8}+a$, the expression can be written as
\begin{equation}\label{eq:hoelder_bound_final}
\sum_{s \in \bbN} \exp\left((-4p-\frac{a^2\kappa}{2}-a\kappa+4a)s \right) s^{-p\beta}\ell(s)^{-1} 
\sum_{z \in H(e^{-4s},M,T)} y^{(a^2\kappa/8-a) \wedge 0} (1+\abs{x}/y)^{-a} .
\end{equation}
We first sum up in $z \in H(e^{-4s},M,T)$. The result is stated in \cref{le:fw_sum_grid}. There will be three cases relevant to us. The first one will give us the desired result for $\kappa > 1$, the second case will apply to $\kappa \in {]0,1[}$, and the third case to $\kappa = 1$.

\textbf{Case 1:} Suppose either $a \ge 1$, $\frac{a^2\kappa}{8}-a > -2$ or $a \in {]-1,1]}$, $\frac{a^2\kappa}{8} > -1$. In that case the sum over $z \in H(e^{-4s},M,T)$ gives us $(e^{-4s})^{-2} = e^{8s}$, and \eqref{eq:hoelder_bound_final} reduces to
\[ \sum_{s \in \bbN} \exp\left((8-4p-\frac{a^2\kappa}{2}-a\kappa+4a)s \right) s^{-p\beta}\ell(s)^{-1} . \]
This sum converges when (recall the definition of $\ell$)
\begin{alignat*}{2}
&& 8-4p-\frac{a^2\kappa}{2}-a\kappa+4a &\le 0 \\
\text{and}\quad && -p\beta &\le -1 .
\end{alignat*}
Fix $a \in {]-\nu, (-2\nu) \vee 0[}$. To maximise $\alpha$, we need to minimise $p$ which is $p = 2-\frac{a^2\kappa}{8}-\frac{a\kappa}{4}+a$, giving us the optimal exponents
\begin{alignat*}{2}
\alpha &= \frac{1}{p}(-\frac{a^2\kappa}{8}-\frac{a\kappa}{8}+a) 
&&= \frac{-\frac{a^2\kappa}{8}-\frac{a\kappa}{8}+a}{2-\frac{a^2\kappa}{8}-\frac{a\kappa}{4}+a} ,\\
\beta &= \frac{1}{p}
&&= \frac{1}{2-\frac{a^2\kappa}{8}-\frac{a\kappa}{4}+a}
\end{alignat*}
provided that we are in Case 1. Optimising over $a$ yields (after a tedious but elementary computation) $a = \frac{16-4\sqrt{\kappa+8}}{\kappa}$ and $\alpha = 1-\frac{\kappa}{24+2\kappa-8\sqrt{8+\kappa}}$.

One can can check that with this choice of $a$, we are indeed in Case 1 when $\kappa > 1$. Finally, we have
\[
p = \frac{(12+\kappa)\sqrt{8+\kappa}-4(8+\kappa)}{\kappa} > 1 .
\]

\textbf{Case 2:} The second relevant case is $a > 1$, $\frac{a^2\kappa}{8}-a < -2$. In that case the sum over $z \in H(e^{-4s},M,T)$ gives us $(e^{-4s})^{\frac{a^2\kappa}{8}-a} = e^{(4a-\frac{a^2\kappa}{2})s}$, and \eqref{eq:hoelder_bound_final} reduces to
\[ \sum_{s \in \bbN} \exp\left((-4p-a^2\kappa-a\kappa+8a)s \right) s^{-p\beta}\ell(s)^{-1} . \]
This sum converges when (recall the definition of $\ell$)
\begin{alignat*}{2}
&& -4p-a^2\kappa-a\kappa+8a &\le 0 \\
\text{and}\quad && -p\beta &\le -1 .
\end{alignat*}
Fix $a \in {]-\nu, (-2\nu) \vee 0[}$. To maximise $\alpha$, we need to minimise $p$ which is $p = -\frac{a^2\kappa}{4}-\frac{a\kappa}{4}+2a$, giving us the optimal exponents
\begin{alignat*}{2}
\alpha &= \frac{1}{p}(-\frac{a^2\kappa}{8}-\frac{a\kappa}{8}+a) 
&&= \frac{1}{2} ,\\
\beta &= \frac{1}{p}
&&= \frac{1}{-\frac{a^2\kappa}{4}-\frac{a\kappa}{4}+2a}
\end{alignat*}
provided that we are in Case 2. Indeed, we can pick such $a$ when $\kappa < 1$, giving us some corresponding $p>1$.

\textbf{Case 3:} $\kappa=1$, $a=4$. In that case the sum over $z \in H(e^{-4s},M,T)$ gives us the extra factor $(e^{-4s})^{-2}\logp(e^{4s}) = e^{8s}s$, and \eqref{eq:hoelder_bound_final} reduces to
\[ \sum_{s \in \bbN} \exp\left((12-4p)s \right) s^{-p\beta+1}\ell(s)^{-1} . \]
The minimal $p$ to make the sum converge is then $p=3$, giving us the exponents $\alpha = \frac{1}{2}$ and $\beta = \frac{2}{3}$.

\begin{remark}
In case $\kappa \le 1$ (i.e.\@ Cases 2 and 3 above), we can get rid of the boundary effect by restricting to the time interval $[t_0,T]$. In that case, by \cref{le:sle_hit_after_1}, it suffices to consider points $z \in H(e^{-4s},M,T)$ with $\Im z \ge \varepsilon$. By \cref{rm:sum_grid_capped}, the sum over such $z$ is bounded by $\varepsilon^{a^2\kappa/8-a+2}e^{8s}$. From here, we can follow Case 1, with an additional factor of $\varepsilon^{a^2\kappa/8-a+2}$.

By \cref{le:sle_hit_after_1}, the probability that this $\varepsilon$ does not suffice is less than $\varepsilon^{4/\kappa-1}$. This implies the result of \cref{thm:main_hoelder} in the case $\kappa\le 1$, and that the Hölder constant has finite moments for $\tilde p < p\,\frac{4/\kappa-1}{4/\kappa-1-a^2\kappa/8+a-2}$. A calculation shows that the latter exponent is still greater than $1$ (even greater than $2.99$).
\end{remark}

\begin{theorem}\label{thm:hoelder}
Let $\kappa \in {]1,8[} \cup {]8,\infty[}$. Let $\varphi(t) = t^{\alpha} (\logp(\frac{1}{t}))^{\beta} \ell(\logp(\frac{1}{t}))^{\beta}$ with $\alpha = 1-\frac{\kappa}{24+2\kappa-8\sqrt{8+\kappa}}$, $\beta = \frac{\kappa}{(12+\kappa)\sqrt{8+\kappa}-4(8+\kappa)}$, and $\ell$ as in \eqref{eq:ell_def}.
Then 
\[ \bbE\left[ \sup_{t_1,t_2 \in [0,T]} \left( \frac{\abs{\gamma(t_1)-\gamma(t_2)}}{\varphi(\abs{t_1-t_2})} \right)^{p} \right] < \infty  \]
for any $p < 1/\beta$. 

In case $\kappa \in {]0,1]}$, there exists $\beta < 1$ such that the same is true with $\varphi(t) = t^{1/2} (\logp(\frac{1}{t}))^{\beta}$.

In particular, there almost surely exists some random $C < \infty$ such that
\[ \abs{\gamma(t_1)-\gamma(t_2)} \le C \varphi(\abs{t_1-t_2}) \]
for all $t_1,t_2 \in [0,T]$.
\end{theorem}

\begin{proof}
We can restrict to the event $\{ \norm{\xi}_{[0,T]} \le M \}$ since all our previous estimates depend on $M$ polynomially (the dependence comes from \cref{le:fw_sum_grid}), whereas the probability of $\bbP( \norm{\xi}_{[0,T]} > M )$ decays exponentially in $M$.

So we are almost reduced to what we have estimated above. Writing $v = \abs{t_1-t_2}^{1/2}$, we have
\begin{multline*}
\left( \frac{\abs{\gamma(t_1)-\gamma(t_2)}}{\varphi(\abs{t_1-t_2})} \right)^p \\
\lesssim \left( \frac{\abs{\gamma(t_1)-\hat f_{t_1}(iv)}}{\varphi(\abs{t_1-t_2})} \right)^p 
+ \left( \frac{\abs{\hat f_{t_1}(iv)-\hat f_{t_2}(iv)}}{\varphi(\abs{t_1-t_2})} \right)^p 
+ \left( \frac{\abs{\gamma(t_2)-\hat f_{t_2}(iv))}}{\varphi(\abs{t_1-t_2})} \right)^p .
\end{multline*}
The first and the third term appear already in \cref{lem:hoelder_bound} which we have bounded by \eqref{eq:hoelder_bound_pq}. Recall that the expression \eqref{eq:hoelder_bound_pq} does not depend on the choice of $t_1,t_2$, and that $\bbE[\text{ eq. \eqref{eq:hoelder_bound_pq} }] < \infty$.

So we are left to estimate the middle term. We show that it is bounded by \eqref{eq:hoelder_bound_pq} as well.

By \eqref{eq:f_diff}, we have
\[
\abs{\hat f_{t_1}(iv)-\hat f_{t_2}(iv)} 
\le \Upsilon_{t_1}(\hat f_{t_1}(iv)) \left( 1 + \left(\frac{\abs{\xi(t_1)-\xi(t_2)}^2}{\abs{t_1-t_2}}\right)^l \right) .
\]

Pick $m \in \bbN$ with $2^{-m} \asymp_2 v$.

As in the proof of \cref{lem:gvar_bound}, we can find $s \in \bbN$ such that $\Upsilon_{t_1}(\hat f_{t_1}(i2^{-m})) \in [e^{-4s},e^{-4(s-1)}]$, and as before, we can replace $\hat f_{t_1}(i2^{-m})$ by $z(t_1,m) \in H(e^{-4s},\norm{\xi},T)$, and we have $Y_{t_1}(z(t_1,m)) \asymp 2^{-m}$ and $\Upsilon_{t_1}(z(t_1,m)) \asymp \Upsilon_{t_1}(\hat f_{t_1}(i2^{-m})) \asymp e^{-4s}$, and $(m,t_1) \in P(z,s)$.

Finally, since $\abs{t_1-t_2} = v^2 \asymp_4 2^{-2m}$, by \cref{lem:brownian_increment}
\[ \left(\frac{\abs{\xi(t_1)-\xi(t_2)}^2}{\abs{t_1-t_2}}\right)^l \le M_n(z,s)^{2l} . \]
Hence,
\[ \begin{split}
&\left( \frac{\abs{\hat f_{t_1}(iv)-\hat f_{t_2}(iv)}}{\varphi(\abs{t_1-t_2})} \right)^p \\
&\quad \le \left( \Upsilon_{t_1}(\hat f_{t_1}(iv)) \left( 1 + \left(\frac{\abs{\xi(t_1)-\xi(t_2)}^2}{\abs{t_1-t_2}}\right)^l \right) \varphi(\abs{t_1-t_2})^{-1} \right)^p \\
&\quad \lesssim \sum_{s \in \bbN} \sum_{z \in H(e^{-4s},M,T)} \sum_{n \in \bbN_0} e^{-4ps} (1+M_n(z,s)^{2pl}) \varphi(Y_{\sigma(S_n)}^2)^{-p} \,1_{S_n(z,s) < \infty} .
\end{split} \]
We see that this sum is also bounded by the expression \eqref{eq:hoelder_bound_pq} except for a factor $(1+M_n(z,s)^{2lp})$. But when taking the expectation, that factor is irrelevant since it is independent of $\mathcal F_{\sigma(S_n)}$ and has finite exponential moments.
\end{proof}

\subsection{The case $\kappa=8$}
\label{se:sle8}

In the case $\kappa = 8$ upper bounding the increment $\abs{\gamma(t)-\gamma(s)}$ as described in \cref{rm:crad_vs_path} is not sufficient any more. The reason is that the radial Bessel process $\hat\theta$ from \cref{se:prelim_bm} has index $\nu = 0$, i.e.\@ it barely misses the boundary. As a consequence, the corresponding Bessel clock has heavy tails (see \cref{pr:bes_clock_critical}). Therefore our bounds on $Y_{\sigma(s)}$ become much worse than in the previous section; in particular, we cannot freely sum them any more. (Such obstacle has been already observed by \cite{al-sle8}.)

A major improvement is given by \cite[Lemma 7.1]{kms-sle48x} (building on \cite[Section 3]{ghm-kpz-sle}). They prove that with high probability, the SLE$_8$ trace fills a ball of radius $\delta^{1+\varepsilon}$ before travelling distance $\delta$. More precisely, if $\abs{\gamma(t)-\gamma(s)} > \delta$, then $B(\gamma(u),\delta^{1+\varepsilon}) \subseteq \gamma[s,t]$ for some $u \in [s,t]$. Notice that in the half-plane capacity parametrisation we have $\Im g_s(\gamma(u)) \le 2\abs{t-s}^{1/2}$ since $2\abs{t-s} = \hcap K_t - \hcap K_s = \hcap g_s(K_t\setminus K_s) \ge \frac{1}{2} (\Im g_s(\gamma(u)))^2$ (the last inequality is \cite[Lemma 1]{lln-hcap}).

In particular, we extract from \cite[Lemma 7.1]{kms-sle48x} the following property which turns out to be the relevant information.

\begin{lemma}\label{le:ball_filled}
Fix $\varepsilon > 0$. For $\delta > 0$ the following event holds with probability at least $1-o^\infty(\delta)$ (where $o^\infty(\delta)$ means a function that decays faster than any power of $\delta$):

If $s < t$, $\abs{\gamma(s)} < 1$, and $\abs{\gamma(t)-\gamma(s)} > \delta$, then there exists $z \in K_t \setminus K_s$ such that $\Upsilon_s(z) \ge \delta^{1+\varepsilon}$ and $Y_s(z) \le  2\abs{t-s}^{1/2}$.
\end{lemma}

\begin{remark}
This lemma removes the need of estimating the total path length $\int_0^v \abs{\hat f_s'(iu)}\, du \asymp \sum_m \Upsilon_s(\hat f_s(iv2^{-m}))$. Instead, we can immediately estimate $\abs{\gamma(t)-\gamma(s)}$ by $\Upsilon_s(\hat f_s(z))^{1/(1+\varepsilon)}$ for some $z$ with $\Im z \le \abs{t-s}^{1/2}$. This differs to the lower bound in \cref{rm:crad_vs_path} by a small exponent. In case $\kappa = 8$, this upper bound performs much better than the sum $\sum_m \Upsilon_s(\hat f_s(iv2^{-m}))$.
\end{remark}

Note that by applying \cref{le:ball_filled} with $\delta_m = 2^{-m}\delta$, $m\in\bbN$, we see that with probability at least $1-o^\infty(\delta)$, the following holds: If $s < t$, $\abs{\gamma(s)} < 1$, and $\abs{\gamma(t)-\gamma(s)} \le \delta$, then there exists $z \in K_t \setminus K_s$ such that $\Upsilon_s(z) \ge \abs{\gamma(t)-\gamma(s)}^{1+\varepsilon}$ and $Y_s(z) \le  2\abs{t-s}^{1/2}$.

We fix a finite time interval $[0,T]$. Observe that the Loewner equation automatically guarantees $\abs{\gamma(s)} \le \norm{\xi}_{\infty;T}+2\sqrt{T}$ for $s \in [0,T]$. Note that $\bbP( \norm{\xi}_{\infty;T} > M) = o^\infty(1/M)$. Therefore, it follows from \cref{le:ball_filled} via scaling that the following event has probability $\bbP(E_{M}) = 1-o^\infty(1/M)$.
\[
E_{M} = \{ \norm{\xi}_{\infty;[0,T]} \le M \} \cap 
\left\{ \begin{array}{cc}
\text{For any } 0 \le s < t \le T, \text{ there exists } z \in K_T \text{ such that } \\
\Upsilon_s(z) \ge M^{-1}\abs{\gamma(t)-\gamma(s)}^{1+\varepsilon} \text{ and } Y_s(z) \le 2\abs{t-s}^{1/2} 
\end{array} \right\} .
\]

As in the case $\kappa \neq 8$, we can restrict to the points on the grid $H(h,\norm{\xi}_\infty, T)$ defined in \eqref{eq:fw_grid}. If $\Upsilon_s(z) \ge \delta^{1+\varepsilon}$, then (by \cref{le:distortion}) we can find $\tilde z \in H(\delta^{1+\varepsilon},\norm{\xi}_\infty,T)$ such that $Y_s(\tilde z) \le \frac{3}{2}Y_s(z)$ and $\Upsilon_s(\tilde z) \ge \frac{27}{160}\Upsilon_s(z)$.

In what follows, we let $\varphi \colon {[0,\infty[} \to {[0,\infty[}$ be a non-decreasing function such that for every $c>1$ there exists $\tilde\Delta_c$ such that $\varphi(ct) \le \tilde\Delta_c \varphi(t)$ for all $t$. Now, repeating the proof of \cref{lem:gvar_bound}, we get the following statement.

\begin{lemma}\label{le:hoelder_bound_8}
Let $M,T > 0$ and $\varepsilon > 0$. There exists $C > 0$ depending on $\varphi,T,\varepsilon$ such that if $E_M$ holds, then
\[
\frac{\abs{\gamma(t_1)-\gamma(t_2)}}{\varphi(\abs{t_1-t_2})} \le C M^{1/(1+\varepsilon)} \sup_{s \in \bbN} \sup_{z \in H(e^{-4s},M,T)} \frac{e^{-4s/(1+\varepsilon)}}{\varphi(Y_{\sigma(s+1,z)}(z)^2)}
\]
for any $t_1,t_2 \in [0,T]$.
\end{lemma}

\begin{remark}
Similarly, the proof shows that if $E_M$ holds, then
\[
\sum_j \psi\left(\frac{\abs{\gamma(t_j)-\gamma(t_{j-1})}}{\varphi(\abs{t_j-t_{j-1}})}\right) 
\le C \sum_{s \in \bbN} \sum_{z \in H(e^{-4s},M,T)} \psi\left(M^{1/(1+\varepsilon)} \frac{e^{-4s/(1+\varepsilon)}}{\varphi(Y_{\sigma(s+1,z)}(z)^2)}\right)
\]
for any partition of $[0,T]$.

In particular, putting $\varphi(t) = 1$ and $\psi(x) = x^{2+\tilde\varepsilon}$, we recover that the SLE$_8$ trace has finite $(2+\tilde\varepsilon)$-variation. With other choices of $\varphi$ and $\psi$, we can interpolate between variation regularity and Hölder-type modulus and get a range of regularity statements.
\end{remark}

Recall from \cref{se:prelim_general,se:prelim_bm} (and explained again in \cref{se:pf_warmup}) that we can write
\[ Y_{\sigma(s)}(z) = y \exp\left( -\frac{2}{\kappa} C_{\kappa(s-s_0)} \right) \]
where $C$ is the radial Bessel clock defined in \cref{se:bessel}, for a radial Bessel process of index $\nu = \frac{1}{2}-\frac{4}{\kappa} = 0$ started at $\hat\theta_{s_0} = \cot^{-1}(x/y)$ (where $s_0 = -\frac{1}{4}\log y$).

Let $\ell$ be as in \eqref{eq:ell_def}. By \cref{pr:bes_clock_critical}, we have that
\[ \begin{split}
\bbP( C_{\kappa(s-s_0)} > re^{16s}s^{4}\ell(s)^2 ) 
&\lesssim (s-s_0+\abs{\log\sin\hat\theta_{s_0}})\,r^{-1/2}e^{-8s}s^{-2}\ell(s)^{-1} \\
&\lesssim \logp(\abs{x}/y) r^{-1/2}e^{-8s}s^{-1}\ell(s)^{-1}
\end{split} \]
and therefore
\[ \bbP( \logp Y_{\sigma(s)}(z)^{-1} > re^{16s}s^{4}\ell(s)^2 ) \lesssim \logp(\abs{x}/y) r^{-1/2}e^{-8s}s^{-1}\ell(s)^{-1} . \]
By the assumption on $\ell$, this expression is summable over $s \in \bbN$ and $z \in H(e^{-4s},M,T)$, and we get with probability $1-O((M\logp M)r^{-1/2})$ that
\begin{equation}\label{eq:k8_Ybound}
\logp Y_{\sigma(s)}(z)^{-1} \le re^{16s}s^{4}\ell(s)^2
\end{equation}
for all $s \in \bbN$, $z \in H(e^{-4s},M,T)$.

Picking $\varphi(t) = (\logp(\frac{1}{t}))^{-1/4+\varepsilon}$ we get from \cref{le:hoelder_bound_8} that with probability $1-O((M\logp M)r^{-1/2})$ we have
\[ \begin{split}
\sup_{t_1,t_2 \in [0,T]}\frac{\abs{\gamma(t_1)-\gamma(t_2)}}{\varphi(\abs{t_1-t_2})}
&\lesssim M^{1/(1+\varepsilon)} \sup_{s \in \bbN} \sup_{z \in H(e^{-4s},M,T)} \frac{e^{-4s/(1+\varepsilon)}}{(re^{16s}s^{4+\varepsilon})^{-1/4+\varepsilon}} \\
&\lesssim M^{1/(1+\varepsilon)} r^{1/4-\varepsilon} .
\end{split} \]
We obtained the following result.

\begin{theorem}
Let $\kappa=8$ and fix $T > 0$. For any $\varepsilon > 0$ there exist $\delta > 0$ and $C < \infty$ (depending on $\varepsilon,T$) such that
\[ \bbP\left( \sup_{t_1,t_2 \in [0,T]} \frac{\abs{\gamma(t_1)-\gamma(t_2)}}{\logp\left(\frac{1}{\abs{t_1-t_2}}\right)^{-1/4+\varepsilon}} > r \right) \le C r^{-2-\delta} . \]
\end{theorem}

\section{Regularity of the uniformising maps}
\label{se:unif_map}

In this section we prove \cref{th:main_unif_map}. We follow the same lines as in \cref{sec:regularity_proofs}. It is even a bit simpler since in the case $\kappa \ge 8$ the process $Y_{\sigma(s,z)}(z)^{-1}$ is monotone in $s$. (In fact, this is crucial for the proof; the process $Y_{\sigma(s,z)}(z)^{-1} 1_{\sigma(s,z) < \infty}$ would not be monotone in case $\kappa < 8$.)

Let $\varphi\colon {[0,\infty[} \to {[0,\infty[}$ be a non-decreasing function such that for every $c>1$ there exists $\tilde\Delta_c$ such that $\varphi(cx) \le \tilde\Delta_c \varphi(x)$ for all $x$. Recall the notations $\Upsilon_t$, $\sigma$, and $H(h,M,T)$ introduced in \cref{se:prelim_general}.

\begin{lemma}\label{le:unif_map_bound}
Let $M,T > 0$. There exists $C > 0$ depending on $\varphi,T$ such that if $\norm{\xi}_{\infty;[0,T]} \le M$, then
\[
\frac{(\Im w)\abs{\hat f_t'(w)}}{\varphi(\Im w)}
\le 
C \sup_{s \in \bbN} \sup_{z \in H(e^{-4s},2M+4\sqrt{T},T)} \frac{ e^{-4s} }{\varphi(Y_{\sigma(s,z)}(z))} 
\]
for any $w \in \bbH$ with $\Im w \le 1$ and $t \in [0,T]$.
\end{lemma}

\begin{proof}
This is proved in the exact same way as \cref{lem:gvar_bound}, now using \cref{le:fw_grid_all} instead of \cref{thm:fw_grid}. Concretely, suppose $\Upsilon_t(\hat f_t(w)) \in [e^{-4s},e^{-4(s-1)}]$. Similarly as before (now arguing by \cref{le:fw_grid_all}), we find $z = z(t,w) \in H(e^{-4s},2\norm{\xi}+4\sqrt{T},T)$ with $\Upsilon_t(z) \asymp e^{-4s}$ and $\abs{Z_t(z)-w} \le \Im w/2$. The rest of the argument is the same.
\end{proof}

To prove \cref{th:main_unif_map}, we bound the moments of the right-hand side of \cref{le:unif_map_bound}. We handle the cases $\kappa=8$ and $\kappa>8$ separately.

Suppose $r>1$ is some large number. For $\kappa=8$, by \eqref{eq:k8_Ybound}, we have with probability $1-O((M\logp M)r^{-1/2})$ that
\begin{equation}\label{eq:k8_Ybound_new}
\logp Y_{\sigma(s)}(z)^{-1} \le re^{16s}s^{4}\ell(s)^2
\end{equation}
for all $s \in \bbN$, $z \in H(e^{-4s},2M+4\sqrt{T},T)$. Hence, we can pick the function $\varphi(x) = (\logp(\frac{1}{x}))^{-1/4}(\logp\logp(\frac{1}{x}))(\ell(\logp\logp(\frac{1}{x})))^{1/2}$. On this event, \eqref{eq:k8_Ybound_new} (together with the monotonicity of $\varphi$ near $0$) imply that the right-hand side of \cref{le:unif_map_bound} is bounded by (a constant times) $r^{1/4}$.

Combining this with the fact that $\bbP(\norm{\xi}_{\infty;[0,T]} \ge M) \lesssim \exp(-cM^2)$, this shows the following result.

\begin{theorem}
Let $\kappa = 8$. Let $\ell$ be as in \eqref{eq:ell_def}. There exists a constant $C>0$ (depending on $\ell,T$) such that
\begin{multline*}
\bbP\left( \sup_{\substack{w \in \bbH \\ \Im w \le 1}} \sup_{t \in [0,T]} \frac{\abs{\hat f_t'(w)}}{(\Im w)^{-1}(\logp(\frac{1}{\Im w}))^{-1/4}(\logp\logp(\frac{1}{\Im w}))(\ell(\logp\logp(\frac{1}{\Im w})))^{1/2}} > r \right) \\
\le Cr^{-2}(\logp r)^{1/2}\logp\logp r . 
\end{multline*}
\end{theorem}

We turn to the case $\kappa > 8$. Let $\psi(x) = x^p$, and $\varphi(x) = x^{2\alpha} (\logp(\frac{1}{x}))^\beta \ell(\logp(\frac{1}{x}))^{1/p}$. We bound the moments of the right-hand side of \cref{le:unif_map_bound} by
\begin{multline}\label{eq:unif_map_pq}
\sup_{s \in \bbN} \sup_{z \in H(e^{-4s},2M+4\sqrt{T},T)} \left( \frac{ e^{-4s} }{\varphi(Y_{\sigma(s,z)}(z))} \right)^p \\
\lesssim
\sum_{s \in \bbN} \sum_{z \in H(e^{-4s},2M+4\sqrt{T},T)} e^{-4ps} \, Y_{\sigma(s)}^{-2p\alpha} \left(\logp Y_{\sigma(s)}^{-1}\right)^{-p\beta} \ell\left(\logp Y_{\sigma(s)}^{-1}\right)^{-1} .
\end{multline}

This is almost the same expression as \eqref{eq:hoelder_bound_pq}, except that we have fixed times $s$ instead of $S_n$. But the calculation following \eqref{eq:Y_mixed_moments} still carries through, with the only difference that we do not have a deterministic bound on $\hat\theta_s$, but we still have $\bbE^{\nu+a}[(\sin\hat\theta_s)^{-a}] < \infty$ since $\nu \ge 0 > -2$. We can follow the same calculation (we are in Case 1 since $\kappa > 8$) which tells us that the optimal exponents that make the right-hand side finite are (with $a = \frac{16-4\sqrt{\kappa+8}}{\kappa}$, $p = 2-\frac{a^2\kappa}{8}-\frac{a\kappa}{4}+a$)
\begin{alignat*}{2}
\alpha 
&= \frac{1}{p}(-\frac{a^2\kappa}{8}-\frac{a\kappa}{8}+a) 
&&= 1-\frac{\kappa}{24+2\kappa-8\sqrt{8+\kappa}} ,\\
\beta 
&= \frac{1}{p} 
&&= \frac{\kappa}{(12+\kappa)\sqrt{8+\kappa}-4(8+\kappa)} .
\end{alignat*}
Combining this with the fact that $\bbP(\norm{\xi}_{\infty;[0,T]} \ge M) \lesssim \exp(-cM^2)$, this shows the following result.

\begin{theorem}
Let $\kappa > 8$. Let $\alpha = 1-\frac{\kappa}{24+2\kappa-8\sqrt{8+\kappa}}$, $\beta = \frac{\kappa}{(12+\kappa)\sqrt{8+\kappa}-4(8+\kappa)}$, and $\ell$ as in \eqref{eq:ell_def}.
Then
\[ \bbE\left[ \sup_{\substack{w \in \bbH \\ \Im w \le 1}} \sup_{t \in [0,T]} \left( \frac{\abs{\hat f_t'(w)}}{(\Im w)^{2\alpha-1}(\logp(\frac{1}{\Im w}))^\beta \ell(\logp(\frac{1}{\Im w}))^{\beta}} \right)^{p} \right] < \infty \]
for any $p < 1/\beta$.
\end{theorem}

%---------------------------------------------------
\bibliographystyle{alpha}
%\bibliography{references.bib}

\end{document}